\documentclass[aap,MSNbibl,seceqn,nameyear,dvips]{arximspdf}
\usepackage[mathscr]{eucal}

\doi{10.1214/11-AAP813} 
\volume{22}
\issue{5}
\pubyear{2012}
\firstpage{1778}
\lastpage{1821}

\makeatletter

\newcommand{\cal}{\mathcal}

\newtheorem{theorem}{Theorem}[section]
\newtheorem{lem}[theorem]{Lemma}
\newtheorem{propn}[theorem]{Proposition}
\newtheorem{cor}[theorem]{Corollary}

\newproclaim{defn}{Definition}

\newcommand{\tF}{\widetilde{F}}
\newcommand{\ntF}{F_\tau}

\newcommand{\rB}[2]{{\cal B}^{(#2)}_{#1}}
\newcommand{\rBo}[1]{{\cal B}^{(#1)}}

\newcommand{\subdiff}{\partial}

\newcommand{\bef}{\Rightarrow}
\newcommand{\comms}{\rightarrow}

\newcommand{\PF}{PF$^+$}
\newcommand{\acci}{\ast\hspace*{-4.3pt}\circ}

\newcommand{\nat}[1]{#1^{\natural}}
\newcommand{\natt}[1]{\hat{#1}}
\newcommand{\cl}[1]{\underline{#1}}

\newcommand{\cF}[1]{{\mathcal F^{(#1)}}}
\newcommand{\cG}[1]{{\mathcal G^{(#1)}}}

\newcommand{\btheta}{\bolds{\theta}}

\newcommand{\cS}{\cal S}

\newcommand{\mbR}{\mathbb{R}}
\newcommand{\tP}{\mathbb{P}}
\newcommand{\tE}{\mathbb{E}}

\newcommand{\speedps}{\Gamma}

\newcommand{\varthetapp}[1]{\vartheta(#1)}
\newcommand{\varthetapps}[1]{\vartheta}
\newcommand{\PhiD}{{\cal D}}
\newcommand{\Phidp}[1]{\PhiD^{+}(#1)}
\newcommand{\psio}{\overline{\psi}}
\newcommand{\psiu}{\underline{\psi}}

\newcommand{\conC}[1]{{\cal C}_{#1}}
\newcommand{\first}{\conC{1}}

\newcommand{\pen}{\conC{K-1}}
\newcommand{\last}{\conC{K}}

\newcommand{\Up}{U}
\newcommand{\Surv}{\mathscr{S}}

\makeatother

\begin{document}
\begin{frontmatter}

\title{Spreading speeds in reducible multitype branching random walk}
\runtitle{Speed in multitype BRW}

\begin{aug}
\author[A]{\fnms{J. D.} \snm{Biggins}\corref{}\ead[label=e1]{J.Biggins@sheffield.ac.uk}}
\runauthor{J. D. Biggins}
\affiliation{University of Sheffield}
\address[A]{School of Mathematics and Statistics\\
University of Sheffield\\
Hicks Building\\
Sheffield, S3 7RH \\
United Kingdom\\
\printead{e1}} 
\end{aug}

\received{\smonth{9} \syear{2008}}
\revised{\smonth{9} \syear{2011}}

%
\begin{abstract}
This paper gives conditions for the rightmost particle in the $n$th
generation of a multitype branching random walk to have a speed, in
the sense that its location divided by $n$ converges to a constant as
$n$ goes to infinity. Furthermore, a formula for the speed is obtained
in terms of the reproduction laws. The case where the collection of
types is irreducible was treated long ago. In addition, the asymptotic
behavior of the number in the $n$th generation to the right of $na$ is
obtained. The initial motive for considering the reducible case was
results for a deterministic spatial population model with several types
of individual discussed by Weinberger, Lewis and Li [\textit{J. Math.
Biol.} \textbf{55} (2007) 207--222]: the speed identified here for the
branching random walk corresponds to an upper bound for the speed
identified there for the deterministic model.
\end{abstract}

%

\end{frontmatter}

\section{Introduction}\label{intro}
The process starts with a single particle located at the origin. This
particle produces daughter particles, which are scattered in $\mbR$, to
give the first generation. These first-generation particles produce
their own daughter particles to give the second generation, and so on.
As usual in branching processes, the $n$th-generation particles
reproduce independently of each other. Particles have types drawn from
a finite set, $\cS$, and the distribution of a particle's family
depends on its type. More precisely, reproduction is defined by a point
process (with an intensity measure that is finite on bounded sets) on
$\cS\times\mbR$ with a distribution depending on the type of the
parent. The first component of the point process determines the
distribution of that child's reproduction point process, its type, and
the second component gives the child's birth position relative to the
parent's. Multiple points are allowed, so that in a family there may be
several children of the same type born in the same place.

Let $Z$ be the generic reproduction point process,
with points $\{(\sigma_i, z_i)\}$, and $Z_\sigma$
the point process (on $\mbR$) of those of type $\sigma$.
Let $\tP_{ \nu}$ and $\tE_\nu$ be
the probability and expectation associated
with reproduction from a parent with type $\nu\in\cS$. Thus,
$\tE_\nu Z_\sigma$ is the intensity measure of the positions of
children of
type $\sigma$ born to a parent of type $\nu$ at the origin.
The usual Markov-chain classification ideas can be used to classify the types:
the type-space is divided, using
the relationship ``can have a descendant of this type,'' into
self-communicating classes,
each of which corresponds to an irreducible multitype branching process.
Two types are in the same class exactly when each can have a descendant,
in some generation, of the other type. A class will be said to
precede another if the first can have descendants in the second,
and then the second will be said to stem from the first.

Let $Z^{(n)}$ be the $n$th-generation point process. Let
$Z^{(n)}_\sigma
$ be the points
of~$Z^{(n)}$ with type $\sigma$. Later,
exponential moment conditions on the intensity measure of $Z$
will be imposed that ensure these are
well-defined point processes
(because the expected numbers in bounded sets are finite).
Let $\cF{n}$ be the information on all families with the parent in a
generation up to and including $n-1$. Hence $Z^{(n)}$ is known when
$\cF
{n}$ is known.
Let $m(-\theta)$ be the nonnegative matrix of the Laplace transforms
of the
intensity measures~$\tE_\nu Z_\sigma$:
\[
(m(\theta))_{ \nu\sigma}=\int e^{\theta z} \tE_\nu
Z_\sigma(dz)=
\tE_\nu
\biggl[\int e^{\theta z} Z_\sigma(dz)
\biggr].
\]
Then it is well known, and verified by induction, that the powers of
the matrix $m$ provide the transforms of the intensity measures $\tE
_\nu Z^{(n)}_\sigma$:
%
%
\begin{equation}\label{powm}
\tE_\nu\biggl[\int e^{ \theta z} Z^{(n)}_\sigma(dz)\biggr]= \int e^{
\theta z} \tE_\nu Z^{(n)}_\sigma(dz)=( m(\theta )^n )_{ \nu\sigma}.
\end{equation}
Let
$\rB{\sigma}{n}$
be the rightmost particle of type $\sigma$ in the $n$th generation,
so that
\[
\rB{\sigma}{n}= \sup\bigl\{z\dvtx z \mbox{ a point of } Z^{(n)}_\sigma\bigr\}
\]
and let $\rBo{n}$ be the rightmost of these.

When the collection of types is irreducible, so that any type can
occur in the line of descent of any type,
and there is a $\phi>0$ such that
%
%
\begin{equation}\label{mgood}
\sup_{\nu,\sigma} (m(\phi))_{ \nu\sigma}< \infty,
\end{equation}
there is a constant $\Gamma$ such that
%
%
\begin{equation}\label{limit}
\frac{\rBo{n}}{n}\rightarrow\Gamma\qquad\mbox{a.s.-}\tP_{ \nu},
\end{equation}
when the process survives. When this holds the speed, starting in $\nu
$, is~$\Gamma$. This result is in Biggins
[(\citeyear{MR0420890}), Theorem 4] and, in
a more general framework
where time is not assumed discrete, in Biggins
(\citeyear{MR1601689}), Section 4.1.
Furthermore, with the\vadjust{\goodbreak} obvious adjustment for periodicity,
the same result holds with $\rB{\sigma}{n}$ in place of $\rBo
{n}$---when the type set is aperiodic this is
in \citet{my-thesis}, Corollary V.4.1.
The theory for the irreducible process also
provides various formulas
for $\Gamma$ in terms of the reproduction process. The question addressed
here is what happens when the set of types is reducible.

Write the transpose of $m$ in
the canonical form of a nonnegative matrix,
described in Seneta (\citeyear{MR0389944,MR719544}), Section 1.2.
This amounts to ordering the rows, and the labels on the classes,
so that when one class stems from another
it is also later in the ordering. Then there are
irreducible blocks, one for each class,
down the diagonal and all other nonzero entries in $m$ are above
this diagonal structure. Having done this, call the first class,
$\first$,
the second $\conC{2}$ up to the final one $\last$.
Intermediate classes need not be totally ordered by ``descends from,''
so their ordering need not be unique.

Any irreducible matrix has a ``Perron--Frobenius'' eigenvalue (which is
positive, is largest in modulus
and has corresponding left and right eigenvectors that are strictly
positive)---see Seneta (\citeyear{MR0389944,MR719544})
or \citet{MR792300}. For $\theta\geq0$,
let $\exp(\kappa_i(\theta))$ be the ``Perron--Frobenius'' eigenvalue of
the $i$th irreducible block, which is infinite when any entry is
infinite. Let $\kappa_i(\theta)=\infty$ for $\theta<0$; this is
just a
device to simplify the formulation, since the development concerns only
the right tails of the measures---left tails and the
consideration of the leftmost particle are just the mirror image. Call
$\kappa_i$ the \PF eigenvalue of the corresponding matrix, which with
these definitions is not necessarily its ``Perron--Frobenius'' eigenvalue
for strictly negative arguments.
As Laplace transforms, the logarithms of the nonzero entries in $m$ are convex.
Then $\kappa_i$ is convex---see Lemma~\ref{analytic} below.

Let $\PhiD( f)$ be the set where the function $f$ is not
$+\infty$, so that
$ \PhiD( f)=\{\theta\dvtx f(\theta)<\infty\}
$. Thus in the irreducible case~(\ref{mgood}) is equivalent to $\PhiD
( \kappa)\cap(0,\infty)\neq\varnothing$. Furthermore,
since each $\kappa
_i$ is convex, $\PhiD( \kappa_i)$ must be an interval in
$[0,\infty)$.
For any two classes $\conC{i}$ and $\conC{j}$
let
\[
\PhiD_{i,j}= \bigcap
\{\PhiD( m_{ \nu\upsilon})\dvtx\nu\in\conC{i},
\upsilon\in\conC
{j}, m_{\nu\upsilon}>0\},
\]
which is the set where all of the entries in $m$ linking $\conC{i}$ to
$\conC{j}$ are finite.
For any set of reals $A$ let $A^+$ be all values either in $A$
or greater than
those in~$A$. Thus, $\Phidp{f}$ has the form $[\varphi,\infty)$
or $(\varphi, \infty)$, depending on whether $f(\varphi)$ is finite
or not.

Without loss of generality, assume that the initial type $\nu$ is in
the first class, $\conC{1}$, and that the speed is sought for a type
$\sigma$ in the final class, $\conC{K}$. Write $i \comms j$ if some
$\nu\in\conC{i}$ can have a child (i.e., an immediate descendant)
with a type in $\conC{j}$ and write $i \bef j$ when
$i$ precedes $j$ so that types in class~$\conC{i}$ can have
descendants in some later generation with types in class $\conC{j}$.
Assume also, again without loss, that every other class stems from the
first and precedes the last.
It is now possible to give a result that illustrates the nature of the
result on speed without the weight of additional notation needed for
its proof or for the results which establish rather more.
%
%
\begin{theorem}\label{overall} Let $\nu\in\conC{1}$, $\sigma\in
\conC{K}$.
Suppose that the process made up of individuals in $\conC{1}$ alone is
supercritical and aperiodic (i.e., the mean matrix is primitive and
has ``Perron--Frobenius'' eigenvalue greater than 1) and survives with
probability 1.
Assume that
%
%
\begin{eqnarray}\label{phiiorder}
&&\mbox{there are }\phi_i \in\PhiD( \kappa_{i}) \qquad\mbox{with }
0<\phi_1\leq\phi_2\leq\cdots\leq\phi_K
\\
\label{phiiorder-two}
&&\mbox{and }
\Phidp{\kappa_{i}} \cap\PhiD( \kappa_{j}) \subset\PhiD
_{i, j} \qquad\mbox{whenever } i \comms j.
\end{eqnarray}
Then
\[
\frac{\rB{\sigma}{n}}{n} \rightarrow \Gamma =\max_{i \bef j}
\inf_{0<\varphi\leq\theta} \max\biggl\{\frac{\kappa_i(\varphi)}
{\varphi},\frac{\kappa_j(\theta)} {\theta}
\biggr\}\qquad\mbox{a.s.-}\tP_{ \nu}.
\]
The conditions~(\ref{phiiorder}) and~(\ref{phiiorder-two}) both
hold when the domain of finiteness of every nonzero entry in the matrix
$m$ has the same nonempty intersection with $[0, \infty)$.
\end{theorem}

This result, other than the actual form of the limit, will be derived
as a by-product of a result on the size of $Z^{(n)}_{\sigma}[na,\infty
)$ described later,
in Theorem~\ref{maintheorem}. That approach to deriving the speed was
used for the one-type process in \citet{MR0464415} and for the
irreducible process in \citet{MR1601689}, Section~4.1. The
comparatively simple formula for the limit here is one of the main
achievements of this study.
One interpretation of this formula for the speed
is the following:
look at each pair of
classes where one precedes the other,
compute the speed as though these were the only classes present, and then
maximize over all such pairs.

It is probably worth being explicit about some of the assumptions that
are not made in Theorem~\ref{overall} and the other main theorems.
First, the point processes $Z$ are not
constrained to have only a finite number
of points. The conditions
do mean that there are only a finite number
of points in any finite interval, but they do
not prevent intervals of the form $(-\infty,a]$
from having an infinite number of points.
Second, classes after the first one do
not have to be supercritical. Third,
classes after the first one do not have
to be primitive. Finally, it is not
assumed that the dispersal in a class
is ``nondegenerate,'' so $\kappa_i$
could be linear in $\theta$ when finite, which for a one-type class
corresponds to a deterministic displacement of the family from the parent.

An initially unexpected phenomenon is contained within Theorem~\ref{overall}.
Its essence can be indicated even
in the reducible two-type case. Suppose type~$a$
can give rise to both type $a$ and type $b$ particles but
type $b$ give rise only to type $b$. Type $a$ or $b$
considered alone forms a one-type branching random walk with speed
$\Gamma_{a}$ or $\Gamma_{b}$, respectively.
At first sight, it seems plausible that,
when $\Gamma_{a}>\Gamma_{b}$, both types spread at speed $\Gamma_{a}$,
driven by the type $a$ particles, and that otherwise, when
$\Gamma_{a} \leq\Gamma_{b}$, the two types move at their own speeds.
This plausible conjecture can be false; it is possible to find examples where,
in the presence of type $a$,
the type $b$
speed can be faster than $\max\{\Gamma_{a},\Gamma_{b}\}$.
The fundamental reason for this ``super-speed'' phenomenon
is that the speed of spread is caused by the interplay
between the exponential growth of the population size
and the exponential decay of
the tail of the dispersal distribution. It is possible
for the growth in numbers of type $a$, through the numbers
of type $b$ they produce, to increase the speed of
type~$b$ from that of a population without type~$a$.
When the type $a$ dispersal distribution
has comparatively light tails, that speed can exceed
also that of type $a$. In this cartoon version,
to get ``super-speed'' we need the population of $a$'s to
grow quickly but the $b$'s to have more chance of
dispersing a long way. This also indicates a
complication. There are two possible sources for
a~comparatively heavy-tailed distribution of the $b$'s.
It could be that the $a$'s, in producing children of
type $b$, disperse them widely, or it could be that
type~$b$'s, in producing $b$'s, produce more spread
than type $a$'s producing $a$'s. Either effect can
influence the speed of the $b$'s. In Theorem~\ref{overall}, (\ref
{phiiorder}) concerns the growth and dispersion within each
irreducible class while~(\ref{phiiorder-two}) controls the dispersion
involved in moving between classes.
The interpretation given above of the formula for the speed shows
that, normally,
the two-type illustration of super-speed is archetypal---there
is no possibility of additional ``cooperation''
from three or more
classes that cannot be exhibited
with just two.

The stimulus for considering this problem was the work of
\citet{MR2322849}, where a deterministic version is discussed and the
phenomenon of ``super-speed,'' which they call ``anomalous spreading
speed,'' is identified---although there the actual speed is not
identified. They also explore the relevance of the phenomenon in a
biological example. There are close relations between these
deterministic models---and also certain continuous-time ones which
involve coupled reaction-diffusion equations---and the branching models
examined here. A~discussion of this connection, which is more than an
analogy, and further illustration of the ``super-speed'' phenomenon based
on applying the results here in the two-type case can be found in the
second half of \citet{MR2744237}.

It turns out that the results for the general case rest on those for
a more restricted class of processes.
A multitype branching process will be called
sequential when each class has children only in its
own class and the next one and there is exactly one
pair of types linking successive classes.
Thus there is just one route through the classes $\conC{1}, \ldots,
\conC{K}$, corresponding to the order of the indices. Also, for
$i=1,\ldots, K - 1$, there is exactly one type in~$\conC{i}$ that can
produce offspring in~$\conC{i+1}$, and just one type of offspring in
$\conC{i+1}$ that it can produce.
The next section describes most of the main results, which concern
sequential processes. The shape of the remainder of the paper will be
indicated in the course of that section and the subsequent one.

\section{Results for the sequential case}
\label{main1}
Throughout this section, the process will be assumed sequential. In the
following one the main results for the general process are given.
Several transformations of functions will be needed to describe the
results. The first is a
version of the Fenchel dual (F-dual) of the function $f$,
given by the convex function
%
%
\begin{equation}\label{Fencheldual}
f^{\ast}(x)=\sup_\theta\{\theta x-f(\theta)\}.
\end{equation}
The second is sweeping strictly positive values to infinity:
let
\[
f^{\circ}(a)=\cases{
f(a), &\quad when $f(a)\leq0$,\cr
\infty,&\quad when $f(a)> 0$.}
\]
Also,
for any function $f$ let
%
%
\begin{equation}\label{speedp}
\Gamma(f)=\inf\{a\dvtx f(a)>0\}.
\end{equation}
Then $\Gamma(f)=\Gamma(f^{\circ})$.
It will also be convenient to have a notation for taking the F-dual and
then sweeping positive values to infinity, so
let
%
%
\begin{equation}\label{defnswFd}
f^{\acci}=(f^{\ast})^{\circ}.
\end{equation}
Various properties of such functions are described in Section
\ref{convexity}. In particular, $f^{\acci}$~is continuous when finite.
The next two results, which are for the case with only one class,
demonstrate why these functions will be useful. Both results are given,
with an indication of their proofs, in Biggins [(\citeyear{MR1601689}),
Section~4.1], and will be discussed further in Section~\ref{primitive},
where various results for the irreducible case that are necessary
preliminaries for the main proofs are obtained.
%
%
\begin{propn} \label{behaviour1}Suppose that there is just one class
of types,
that the exponential moment condition~(\ref{mgood}) holds and that
the matrix $m$ is primitive with \PF eigenvalue $\kappa$. Let
$\Up$ be the upper end-point of the interval on which $\kappa^{\ast
}$ is
finite. Then, for $a \neq\Up$,
%
%
\begin{equation}\label{meanexpgrowthub1}
\lim_n \frac{1}{n}\log\bigl(\tE_\nu Z^{(n)}_\sigma[na, \infty)
\bigr)
=-\kappa^{\ast}(a).
\end{equation}
\end{propn}
%
%
\begin{propn} \label{supercrit} Under the conditions of Proposition
\ref
{behaviour1} and the additional assumption that the process is
supercritical [i.e., $\kappa(0)>0$] and survives with probability 1,
%
%
\begin{equation}\label{expgrowthub2}
\lim_n \frac{1}{n}\log\bigl(Z^{(n)}_\sigma[na, \infty) \bigr) = -
\kappa^{\acci}(a) \qquad\bigl(\mbox{$=$}(\kappa^{\ast})^{\circ}(a) \bigr)
\qquad\mbox{a.s.-}\tP_{ \nu}
\end{equation}
for $a\neq\Gamma(\kappa^{\ast})$ and
\[
\frac{\rB{\sigma}{n}}{n}\rightarrow\Gamma(\kappa^{\ast})=\Gamma
(\kappa^{\acci}) \qquad\mbox{a.s.-}\tP_{ \nu}.
\]
\end{propn}

In this case, there is a simple relationship between the behavior of
$Z^{(n)}_{\sigma}[na,\break\infty)$ and its expectation.
When the expectation decays (geometrically) in~(\ref
{meanexpgrowthub1}) the actual numbers, described by (\ref
{expgrowthub2}), are
ultimately zero, leading to the limit there being infinite (which
explains the sweeping to infinity). On the other hand, when expected
numbers grow the actual numbers grow in the same way. Thus the
``expectation-speed'' and the ``almost-sure-speed'' are the same [and are
both $\Gamma(\kappa^{\ast})$].
In the reducible process this need not be so---the
``expectation-speed''
can overestimate the ``almost-sure-speed.'' The discussion here will
concentrate on the ``almost-sure-speed,'' but expected numbers, which are
easier to study,
will be considered briefly in Section~\ref{expnumbers},
mainly to illustrate the point just made.

The result on the speed in Proposition~\ref{supercrit} is a consequence
of the asymptotic behavior of $n$th-generation numbers in intervals of
the form $(-\infty, na]$. The same basic approach is used to study
reducible sequential processes. There are two
parts to this: showing that a suitable function forms a lower bound
and then showing that it also forms an upper bound. As might be
anticipated from the role of the moment condition~(\ref{mgood}) in the
irreducible case, conditions on the finiteness of the entries in $m$
are needed.
For the simplest lower bound these conditions will only concern the
entries in the irreducible blocks of $m$, as in
(\ref{phiiorder}). But for the upper bound the ``off-diagonal''
entries have to be controlled too, leading to conditions like (\ref
{phiiorder-two}).
The basic idea for obtaining both bounds is to use induction on the
number of classes, with the formula for the bounds being given by
suitable recursions.

Certain properties of the limit $\kappa^{\acci}$ in (\ref
{expgrowthub2}), which is
a rate function in the large deviations' sense,
are sufficiently important here to merit a name.

%
\begin{defn} \label{defr-function}A function will be called an $r$-function
if it is increasing and convex, takes a value in
$(-\infty,0)$, is continuous from the left and is infinite when
strictly positive.
\end{defn}

Whenever $r$ is an $r$-function $\Gamma(r)>-\infty$.
Lemma~\ref{r-funct} shows that $\kappa^{\acci}$ is an \mbox{$r$-function}.

The next theorem, which is proved in Section~\ref{lower}, gives a lower
bound on the numbers, and hence on the speed.
A notation for the convex minorant is needed.
For any two functions $f$ and $g$, let ${\mathfrak C}[f,g]$
be the greatest lower semi-continuous convex function
beneath both of them. (The restriction to lower semi-continuous
functions only affects values at the end-points of the set on which a convex
function is finite.)
%
%
\begin{theorem} \label{prelimmaintheorem} Consider a sequential process
with $K$ classes, $\conC{1}, \ldots, \conC{K}$, with
corresponding \PF eigenvalues $\kappa_1,\ldots,\kappa_K$ and in which
$\conC{1}$, considered alone, is
primitive, supercritical and survives with probability 1.
Assume that~(\ref{phiiorder}) holds.
Define $r_i$ recursively:
%
%
\begin{equation}
\label{firstrecur}\qquad
r_1= \kappa^{\acci}_1 \qquad\bigl(\mbox{$=$}(\kappa^{\ast}_1)^{\circ}
\bigr);\qquad r_{i}={\mathfrak C}[r_{i-1},\kappa^{\ast}_i]^{\circ}
\qquad\mbox{for }
i=2,\ldots, K.\vadjust{\goodbreak}
\end{equation}
Then
for $\nu\in\conC{1}$, $\sigma\in\conC{K}$ and $a \neq\Gamma(r_K)$
%
%
\begin{eqnarray}\label{keyresultlb}
\liminf\frac{1}{n} \log
\bigl(Z^{(n)}_{\sigma}[na,\infty)
\bigr) &\geq&- r_K(a) \qquad\mbox{a.s.-}\tP_{ \nu},
\\
%
%
\label{seqspeedlb}
\liminf_n \frac{\rB{\sigma}{n}}{n} &\geq&
\Gamma(r_K) \qquad\mbox{a.s.-}\tP_{ \nu}
\end{eqnarray}
and $r_K$ is an $r$-function.
\end{theorem}

The first complement to this lower bound is presented next. Once
additional ideas have been introduced, Theorem~\ref{f=g} will give the
same conclusions under weaker conditions.
%
%
\begin{theorem} \label{maintheorem}
In the setup and conditions of Theorem~\ref{prelimmaintheorem},
suppose also that, for $i=1,2,\ldots,K - 1$,
%
%
\begin{equation}\label{off-diag-cond}
\biggl( \bigcap_{j \leq i} \Phidp{\kappa_j}\biggr)\cap\PhiD
( \kappa_{i+1}) \subset\PhiD_{i, i+1}.
\end{equation}
Then
{\renewcommand{\theequation}{Nu}
\begin{equation}
\label{keyresult}
\frac{1}{n} \log
\bigl(Z^{(n)}_{\sigma}[na,\infty) \bigr) \rightarrow-r_K(a)
\qquad\mbox{a.s.-}\tP_{ \nu}
\end{equation}}

\noindent for $a \neq\Gamma(r_K)$, and
{\renewcommand{\theequation}{Sp}
\begin{equation}
\label{seqspeed}
\frac{\rB{\sigma}{n}}{n} \rightarrow\Gamma(r_K)
\qquad\mbox{a.s.-}\tP
_{ \nu}.
\end{equation}}
\end{theorem}

The condition~(\ref{phiiorder}) ensures that the set on the left in
(\ref{off-diag-cond}) contains~$\phi_i$, and so is not empty.
Note that~(\ref{phiiorder}) and~(\ref{off-diag-cond}) just involve
comparing the domains of finiteness
of the entries in $m$. Hence these conditions are easily
applied in the general (nonsequential) case. Note too that (\ref
{phiiorder-two}) in Theorem~\ref{overall} is a stronger assumption
than~(\ref{off-diag-cond}) in this theorem.

To describe the remaining results in this section, one further transformation
is needed. As can be seen from Proposition~\ref{supercrit},
the critical function when looking at actual numbers in the first class
is $\kappa^{\acci}$ (rather than $\kappa^{\ast}$).
Typically, there will be a
$\vartheta\in(0,\infty)$ such that for $a \leq\Gamma(\kappa
^{\ast})$
\[
\kappa^{\ast}(a)=\sup_\theta\{\theta a-\kappa(\theta)\}=\sup
_{\theta
\leq\vartheta} \{\theta a-\kappa(\theta)\}.
\]
Then, with
$\natt{\kappa}(\theta)=\kappa(\theta)$ for $\theta\leq\vartheta
$ and
$\natt{\kappa}(\theta)= \theta\Gamma(\kappa^{\ast}) $ for
$\theta>
\vartheta$, it turns out that $\kappa^{\acci}$ is the F-dual of
$\natt
{\kappa}$, that is, $\kappa^{\acci}=(\natt{\kappa})^{\ast}$.
Thus,
in examining how actual numbers in the first class influence numbers in
the second,
$\natt{\kappa}$ should replace~$\kappa$. This means that the shape of
$\kappa$
only matters up to a certain point, after which it is replaced by a
suitable linear function. The details of $\kappa$ beyond this point
have become irrelevant because they only influence $\kappa^{\ast}$ at
positive values, which are swept to infinity.\vadjust{\goodbreak}

Although this motivation is on the right lines, it turns out that the
actual definition
of the transformation is better framed somewhat differently in order to
cover all cases.
It will also be useful to have a name for the class of functions the
transformation will apply to. Under the conditions of Proposition \ref
{behaviour1}, $\kappa$ satisfies the next definition.
%
%
\begin{defn}\label{defk-convex}
A function is
$k$-convex if it is convex, finite for some $\theta>0$ and infinite for
all $\theta<0$.
\end{defn}

The pointwise supremum of a collection of convex functions is convex, and
that of a collection of monotone functions is monotone.
Hence, for $k$-convex $f$, it makes sense to
define
$\nat{f}$ to be the maximal convex function
such that $\nat{f}\leq f$ and $\nat{f}(\theta)/\theta$
is monotone decreasing in $\theta\in(0,\infty)$. This function will be
identically minus infinity if there are no functions satisfying the constraints.
Now let
%
%
\setcounter{equation}{9}
\begin{equation}\label{defnvartheta}
\varthetapp{f}=\sup\{\theta\dvtx f(\theta)=\nat{f}(\theta)\},
\end{equation}
where it is possible that $\varthetapp{f}=\infty$.
Proposition~\ref{sw-nat} will show that, in the typical case,
$\nat{\kappa}(\theta)$ is just the straight line $ \theta\Gamma
(\kappa^{\ast})$ for $\theta>\varthetapp{\kappa}$, and that line
is the
tangent to $\kappa$ at $\varthetapp{\kappa}$,
which connects this definition with the motivation offered in the
previous paragraph.

An alternative recursion for the $r$-functions defined by~(\ref{firstrecur})
in Theorem~\ref{prelimmaintheorem}
turns out to be more useful when considering upper bounds. This
alternative recursion is given in the next result. Let
$
{\mathfrak M}[f,g](\theta)= \max\{f(\theta),\break g(\theta)\}$.
%
%
\begin{propn}\label{alternativerec} Assume that~(\ref{phiiorder}) holds.
Define $f_i$ recursively:
%
%
\begin{equation}
\label{secondrecur}
f_1=\kappa_1;\qquad f_i={\mathfrak M}[\nat{f}_{i-1},\kappa_i]
\qquad\mbox{for }i=2,\ldots, K.
\end{equation}
Then $(\nat{f}_i)^{\ast}=f^{\acci}_i={\mathfrak M}[\nat
{f}_{i-1},\kappa_i]^{\acci}=r_i$.
\end{propn}

This is proved in Section~\ref{furtherconv}, along with a variety of
convexity results that contribute to deriving formulas for the speed.
The issues surrounding convexity are more complicated than might be
expected on the basis of the known results for the irreducible case.
For example, it is easy to construct (reducible) two-type examples
where $f_2$ and $r_2$ have properties that cannot arise in the one-type
(or irreducible) case. In particular, there are examples where $f_2$ is
linear (only) on a finite or a semi-infinite interval and where $r_2$
is linear (only) on a finite interval.

The notation has now been established to state a result giving
(\ref{keyresult}) and hence~(\ref{seqspeed})
in Theorem~\ref{maintheorem} under weaker conditions. The aim was to
make these conditions as general as is practicable, but that does mean
they are also quite complex. In Theorem~\ref{bestspeed},
(\ref{seqspeed}) will be established under yet weaker conditions.
Let $\psiu_{i}=\inf\PhiD_{i,i+1}$ and
$\psio_i=\sup\PhiD_{i,i+1}$.
%
%
\begin{theorem}\label{f=g}\label{altermaintheorem}
In the setup and conditions of Theorem~\ref{prelimmaintheorem},
suppose that~(\ref{phiiorder}) holds and that for $i=1,2,\ldots,K
- 1$,
%
%
\begin{equation}\label{newphiiorder}
\mbox{there are }\phi_{i,i+1} \in\PhiD_{i,i+1} \qquad\mbox{with }0<
\phi_i
\leq\phi_{i,i+1}\leq\phi_{i+1}.
\end{equation}
Let $f_i$ be as defined at~(\ref{secondrecur}).
Suppose that, for $i=1,2,\ldots,K - 1$,
%
%
\begin{equation}\label{eitheror1}\hspace*{28pt}
\mbox{either }
\kappa_{i+1}(\theta) \geq\theta\bigl(\nat{f}_i(\psio_i)/\psio_i\bigr)
\mbox{ for }\theta\in[\psio_i, \infty)\quad\mbox{or}
\quad\varthetapp{f_i}\leq
\psio_i
\end{equation}
and
%
%
\begin{equation}\label{ii}
\bigcap_{j \leq i} \Phidp{\kappa_j} \cap\PhiD( \kappa
_{i+1}) \subset
[\psiu_{i}, \infty).
\end{equation}
Then~(\ref{keyresult}) and~(\ref{seqspeed}) hold.
\end{theorem}

Complementing the lower bound in Theorem~\ref{prelimmaintheorem}
is a two-stage process, involving first deriving an upper bound and
then giving conditions for it to equal the lower bound.
The first stage is covered by the next result; its proof is in Section
\ref{upper}. Let $I(A)$ be the indicator function of $A$ and let
\[
\chi_{i}=-\log I(\PhiD_{i-1,i}) \qquad\mbox{for }i=2,\ldots,K,
\]
so that $\chi_{i}$ is zero on $\PhiD_{i-1,i}$ and infinity otherwise.
%
%
\begin{theorem}\label{newub}
Make the same assumptions as in Theorem~\ref{prelimmaintheorem}.
Define $g_i$ recursively:
%
%
\begin{equation}
\label{thirdrecur}
g_1=\kappa_1;\qquad g_i={\mathfrak M}[\nat{(\nat{g}_{i-1} + \chi
_i)},\kappa_i]
\qquad\mbox{for } i=2,\ldots,K.
\end{equation}
Then
%
%
\begin{equation}\label{newubeq}
\limsup_n \frac{1}{n} \log
\bigl(Z^{(n)}_{\sigma}[na,\infty)
\bigr) \leq-g^{\acci}_K(a) \qquad\mbox{a.s.-}\tP_{ \nu}
\end{equation}
and
%
%
\begin{equation}\label{seqspeedub}
\limsup_n \frac{\rB{\sigma}{n}}{n} \leq\Gamma(g^{\ast}_K)
\qquad\mbox{a.s.-}\tP_{ \nu}.
\end{equation}
Furthermore,
$-g^{\acci}_K(a)< \infty$ for all $a$
if~(\ref{phiiorder}) holds and~(\ref{newphiiorder}) holds for
$i=1,2,\ldots,K - 1$.
\end{theorem}

A key point from Proposition~\ref{alternativerec}, for the formulation
of the rest of the results in this section, is that
$(f^{\ast}_K)^{\circ}= f^{\acci}_K=r_K$.
Using this,
and comparing~(\ref{keyresultlb}) and~(\ref{seqspeedlb})
with~(\ref{newubeq}) and~(\ref{seqspeedub}), immediately gives the
following corollary.
%
%
\begin{cor}\label{newmain}
Make the same assumptions as in Theorem~\ref{prelimmaintheorem}. Then
(\ref{keyresult}) holds if $g^{\acci}_K=f^{\acci}_K$
and~(\ref{seqspeed}) holds if $\Gamma(f^{\ast}_K)=\Gamma(g^{\ast}_K)$.
\end{cor}

Thus, in the light of this corollary, proving Theorems \ref
{maintheorem} and~\ref{altermaintheorem} will entail showing that
the conditions imposed imply that $g^{\acci}_K=f^{\acci}_K$.
This is done in Section~\ref{firstproof}.\vadjust{\goodbreak}

It is possible that $\Gamma(g^{\ast}_K)
=\Gamma(f^{\ast}_K)$ even though $g^{\acci}_K$ and $f^{\acci}_K$
do not agree everywhere. Then the speed would be given through
(\ref{seqspeed}) of Theorem~\ref{maintheorem}, even though the
behavior of the numbers was not described by~(\ref{keyresult}).
To investigate this possibility,
alternative formulas for $g^{\acci}_K$ and for $f^{\acci}_K$
and their associated speeds are important.
Those formulas are given next.
The formula for $\Gamma(f^{\ast}_K)$ is critical in establishing the
simpler one given in Theorem~\ref{overall}.
Also, the formula for $\Gamma(f^{\ast}_K)$
is the same one that is obtained as the upper bound on the speed in a
deterministic model by Weinberger, Lewis and Li
[(\citeyear{MR2322849}), Proposition 4.1], so their bound can be
simplified, too.

The conventions that
$\PhiD_{0,1}=(0,\infty)$ and $\psio_{K}=\infty$ are now adopted.
It is worth noting that in~(\ref{gform1}) $\theta_K$ is fixed, but in
(\ref{speed-g}) it is one of the free variables in the optimization.
%
%
\begin{theorem}\label{gformullem1} For a sequential process as described
in Theorem~\ref{prelimmaintheorem}, let $g_K$ be given by (\ref
{thirdrecur}).
Then, for $0<\theta_K \in\PhiD^+_{K - 1,K}$,
%
%
\begin{equation}\label{gform1}
\frac
{g_K(\theta_K)}
{\theta_K}
=
\inf\biggl\{
\max_i\biggl\{\frac{\kappa_i(\theta_i)}
{\theta_i}\biggr\}
\dvtx
\theta_1 \leq\theta_2\leq
\cdots\leq\theta_K,
\theta_{i} \in\PhiD^+_{i-1,i},
\theta_{i} \leq\psio_{i}\biggr\}\hspace*{-40pt}
\end{equation}
and ${g_K(\theta_K)}=\infty$ for $0<\theta_K \notin\PhiD^+_{K -
1,K}$. Furthermore,
%
%
\begin{equation}\label{speed-g}
\Gamma(g^{\ast}_K)=\inf\biggl\{
\max_i\biggl\{\frac{\kappa_i(\theta_i)}
{\theta_i}\biggr\}
\dvtx
\theta_1 \leq\theta_2\leq
\cdots\leq\theta_K,
\theta_{i} \in\PhiD^+_{i-1,i},
\theta_{i} \leq\psio_{i}\biggr\}.\hspace*{-40pt}
\end{equation}
Let $f_K$ be given by~(\ref{secondrecur}).
These formulas hold with $f_K$ in place of $g_K$ on replacing $\PhiD
_{i,i+1}$ by $(0, \infty)$
(and $\psio_{i}$ by $\infty$)
for $i=1,2,\ldots,K - 1$.
\end{theorem}

Now, asking when the formulas for $\Gamma(g^{\ast}_K)$ and $\Gamma
(f^{\ast}_K)$ give the same result---that is, when the extra restrictions
in the optimization associated with the formula for $\Gamma(g^{\ast}_K)$
make no difference---leads to the following theorem. Both it and the
previous theorem are proved in
Section~\ref{formulaforGamma}, where a little more is also said about
formulas for $\Gamma(f^{\ast}_K)$.
%
%
\begin{theorem}\label{bestspeed}In the setup and conditions of Theorem
\ref
{prelimmaintheorem}, suppose (\ref{newphiiorder}),~(\ref{eitheror1})
and $\varthetapp{\kappa_{i+1}} \geq\psiu_{ i}$
all hold for $i=1,2,\ldots,K - 1$.
Then
$\Gamma(g^{\ast}_K)
=\Gamma(f^{\ast}_K)$ and~(\ref{seqspeed}) holds.
\end{theorem}

Theorem~\ref{newub} also
raises the question of whether
the upper bound there, when it is actually larger than the lower bound
in Theorem~\ref{prelimmaintheorem},
can be matched by a corresponding lower
bound. A full study of this is not attempted,
but some key results are given in the final section of the paper.

\section{From sequential to general}
\label{sequent}
The main idea here is to explain how in the general case
the number of particles of a specified type can be
decomposed using a finite collection of sequential branching processes.
Consider $\sigma\in\last$. Each particle of type $\sigma$ can be
labeled by the classes
that arise in its\vadjust{\goodbreak} ancestry, tracing back to the initial ancestor
in $\first$, and then by the particular types that link the successive
classes. This label will be called its genealogical type.
Thus, for example, the branching process arising from
\[
m= \pmatrix{
m_{11}&m_{12}&m_{13}&m_{14}\cr
0& m_{22}&0&m_{24}\cr
0& 0&m_{33}&m_{34}\cr
0& 0&0&m_{44}}
\]
contains exactly three routes through the classes from the first class
to the fourth, arising from
\[
\pmatrix{ m_{11}&m_{14}\cr0&m_{44}},\qquad \pmatrix{ m_{11}&m_{12}&0\cr0&
m_{22}&m_{24}\cr0&0&m_{44}} \quad\mbox{and}\quad \pmatrix{ m_{11}&m_{13}&0 \cr
0&m_{33}&m_{34}\cr0&0&m_{44}},
\]
and each particle in the final class arises from a line of descent
following one of these three.
For the second phase of the decomposition, each nonzero entry in $m_{14}$
specifies a different type within the first route.
Similarly, a pair of nonzero entries,
one drawn from $m_{12}$ and the other from $m_{24}$,
specifies a type within the second route.

Slightly more formally,
let $\ell$ be a label for genealogical type (so $\ell$ records which
classes occur in the
ancestry and which pairs of types link classes in that ancestry).
Now let $(\sigma, \ell)$ be an augmented type
that indicates those of type $\sigma$
with genealogical type
$\ell$. There are only a finite number of different genealogical types,
and, by definition,
%
%
\begin{equation}\label{seqdecomp}
Z^{(n)}_{\sigma}[na,\infty)=\sum_\ell
Z^{(n)}_{\sigma, \ell}[na,\infty).
\end{equation}
Furthermore, each genealogical type corresponds to a sequential
branching process embedded within the original one.

The next two results follow by straightforward argument from the
decomposition~(\ref{seqdecomp}) and the continuity
of $r$-functions when finite.
Note that the minimum of convex
functions need not be convex,
and so $r$ in this theorem
need not be convex, and hence need not be an $r$-function,
but it will share in the other properties of an $r$-function.
%
%
\begin{theorem} \label{genealogy}
Suppose that, for each $\ell$, there is an $r$-function, $r_\ell$ such that
\[
n^{-1} \log
\bigl(Z^{(n)}_{\sigma, \ell}[na,\infty)
\bigr) \rightarrow-r_\ell(a) \qquad\mbox{a.s.-}\tP_{ \nu}
\]
for all $a \neq\Gamma(r_\ell)$.
Then
\[
n^{-1}\log
\bigl(Z^{(n)}_{\sigma}[na,\infty)
\bigr) \rightarrow- r(a)=-\min_\ell\{r_\ell(a)\}
\qquad\mbox{a.s.-}\tP
_{ \nu}
\]
for all $a \neq\Gamma(r)$ and
\[
n^{-1}\rB{\sigma}{n}\rightarrow\Gamma(r) \qquad\mbox{a.s.-}\tP_{
\nu}.
\]
\end{theorem}
%
%
\begin{theorem} \label{genealogy1}
Suppose that for each $\ell$
%
%
\begin{equation}\label{ellspeed}
n^{-1}\rB{\sigma,\ell}{n} \rightarrow\Gamma_\ell
\qquad\mbox{a.s.-}\tP
_{ \nu}.
\end{equation}
Then
\[
n^{-1}\rB{\sigma}{n} \rightarrow\Gamma= \max_{\ell}\Gamma_\ell
\qquad\mbox{a.s.-}\tP_{ \nu}.
\]
\end{theorem}

Obviously Theorems~\ref{genealogy} and~\ref{genealogy1} can be applied
to get the overall speed when~(\ref{keyresult}) and~(\ref{seqspeed}),
respectively,
hold for every embedded sequential process.
The next result shows that this overall speed is often
not as difficult to calculate as
at first appears. Its proof will be described in Section~\ref{simplespeed}.
%
%
\begin{theorem}\label{pairwiseformula} Suppose that
(\ref{ellspeed}) holds for each embedded sequential process with
$\Gamma_\ell=\Gamma(r_\ell)$ and its associated $r_\ell$ given by the
recursion~(\ref{firstrecur}) in Theorem~\ref{prelimmaintheorem}. Let
$\Gamma$ be the maximum speed obtained as in Theorem~\ref{genealogy1}.
Then
\[
\Gamma= \max_{i \bef j} \{ \Gamma({\mathfrak C}[\kappa
_i^{\acci},\kappa_j^{\ast}])\} =\max_{i \bef j}
\inf_{0<\varphi\leq\theta} \max\biggl\{\frac{\kappa_i(\varphi)}
{\varphi},\frac{\kappa_j(\theta)} {\theta} \biggr\}.
\]
\end{theorem}
\begin{pf*}{Proof of Theorem~\ref{overall}}
The conditions ensure that Theorem~\ref{maintheorem} holds for each
embedded sequential process. Then Theorem~\ref{pairwiseformula} gives
the result.
\end{pf*}

\section{Preliminaries}\label{convexity}

The section introduces various notation and gives some preliminary
results on convexity, drawing heavily on other sources. Further
convexity results that are more particular to this study will be
obtained in later sections.

A convex function is called proper when it is finite somewhere.
A proper convex function is called closed when it is lower
semi-continuous---see Rockafellar [(\citeyear{MR0274683}),
Section 7, page 52] for a full
discussion.
For a convex function on $\mbR$ that is finite
on a nonempty interval, this is the same as demanding
continuity from within at the endpoints of its
domain of finiteness.
The closure $\cl{f}$ of the proper
convex function $f$ on $\mbR$ is obtained
by adjusting the values of $f$ at these endpoints to make it closed.
Thus $\cl{f} \leq f$. By definition, an $r$-function is proper and closed
and so at first sight the nature of the results might suggest that
attention could be restricted throughout to closed convex functions.
However, this is not so.
By using the off-diagonal entry in~$m$, it is easy to construct
(reducible) two-type examples where $g_2$ [given by the recursion
(\ref{thirdrecur})] is not closed
(by being bounded on an open interval but infinite at one of its endpoints).
%
%
\begin{lem}\label{kh}\label{firstconvlem}
\textup{(i)}
\label{firstconvlemi} When $f$ is convex, $f^{\ast}$ is a closed
convex function, as is
$f^{\acci}$ provided it is finite somewhere, and $(f^{\ast})^{\ast
}=\cl{f}$.\vspace*{-6pt}
{
\renewcommand\thelonglist{(ii)}
\renewcommand\labellonglist{\thelonglist}
\begin{longlist}
\item\label{firstconvlemii} If $f$ and $g$ are convex functions,
then so is ${\mathfrak M}[f,g]$ and,
provided ${\mathfrak M}[f,g]$ is finite somewhere,
$
{\mathfrak M}[f,g]^{\ast}={\mathfrak C}[f^{\ast},g^{\ast}]$.\vadjust{\goodbreak}
\end{longlist}
}
\end{lem}
\begin{pf} The first part is all contained in Rockafellar
[(\citeyear{MR0274683}), Theorem~12.2],
except for the claim about $f^{\acci}$, which follows easily from its
definition at~(\ref{defnswFd}).
The first part of~\ref{firstconvlemii} follows directly from the definitions
and the second is in \citet{MR0274683}, Theorems 9.4, 16.5.\vspace*{-2pt}
\end{pf}
%
%
\begin{lem}\label{covex-d}
When $f$ is $k$-convex
(as introduced in Definition~\ref{defk-convex}):
{\renewcommand\thelonglist{(\roman{longlist})}
\renewcommand\labellonglist{\thelonglist}
\begin{longlist}
\item\label{covex-di}$f^{\ast}(a)>-\infty$ for all $a$;
\item\label{covex-dii}$f^{\ast}(a) \rightarrow\infty$ as $a
\uparrow
\infty$ and
$\Gamma(f^{\ast})<\infty$;
\item\label{covex-diii}$f^{\ast}$ is increasing;
\item\label{covex-div}$f^{\ast}(a)<\infty$ for some $a$;
\item\label{covex-dv}$f^{\ast}(a) \rightarrow- \cl{f}(0)$ as $a
\downarrow-\infty$;
\item\label{covex-dvi} $\Gamma(f^{\ast})>-\infty$ if and only if
$\cl
{f}(0)>0$.\vspace*{-2pt}
\end{longlist}}
\end{lem}
\begin{pf}
When $f(\phi)<\infty$,
$
f^{\ast}(a)\geq\phi a-f(\phi)>-\infty
$
giving~\ref{covex-di}, and, since $\phi>0$, letting $a \uparrow
\infty
$ gives~\ref{covex-dii}.
Furthermore, because $f(\theta)=\infty$ for $\theta<0$,
\[
f^{\ast}(a)=\sup_{\theta}\{\theta a -f(\theta)\}=\sup_{\theta\geq
0}\{
\theta a -f(\theta)\}
\leq\sup_{\theta\geq0}\{\theta a' -f(\theta)\},
\]
when $a' \geq a$, so $f^{\ast}$ is increasing in $a$.
Since $f$ is finite and convex there must be
finite $A$ and $B$ such that
$
f(\theta)\geq A \theta-B
$ for all $\theta$
and then \mbox{$f^{\ast}(A) \leq B$}, giving~\ref{covex-div}.
Part~\ref{covex-dv} follows
from Lemma~\ref{kh}(i) and \citet{MR0274683}, Theorem~27.1(a).
Part~\ref{covex-dvi} follows directly from~\ref{covex-diii},
\ref{covex-dv} and the definition of~$\speedps$.\vspace*{-2pt}
\end{pf}

The next result gives properties of $\kappa$ arising from irreducible $m$.
It is worth stressing that part~\ref{analyticiii} includes claims about
one-sided derivatives at the endpoints of $\PhiD( \kappa)$.\vspace*{-2pt}
%
%
\begin{lem}\label{analytic}
Suppose $\kappa$ is the \PF eigenvalue of an irreducible $m$ and that
(\ref{mgood}) holds:
{\renewcommand\thelonglist{(\roman{longlist})}
\renewcommand\labellonglist{\thelonglist}
\begin{longlist}
\item\label{analytici}$\PhiD( \kappa)$ is a (possibly degenerate)
interval containing the $\phi$ in~(\ref{mgood});
\item\label{analyticii}$\kappa$ is $k$-convex;
\item\label{analyticiii}$\kappa$ is continuous on the closure of
$\PhiD( \kappa)$, differentiable on
$\PhiD( \kappa)$ and analytic on its interior;
\item\label{analyticiv}$\kappa$ is closed.\vspace*{-2pt}
\end{longlist}}
\end{lem}
\begin{pf} Clearly~(\ref{mgood}) implies that $\kappa(\phi)<\infty$.
For convexity,
see \citet{MR0138632}, \citet{MR0126886}
and
\citet{MR0389944}, Theorem 3.7.
Part~\ref{analyticii} follows immediately from this and~(\ref{mgood}).
For analyticity on the interior, which is a
straightforward application of the implicit function theorem,
see Miller [(\citeyear{MR0126886}), Theorem 1(a)],
Lancaster and Tismenetsky [(\citeyear{MR792300}), Theorem 11.5.1]
or \citet{MR2156555}, Theorem 1(i).
Each entry in $m$ is continuous on the
closure of the set
where it is finite and so the same
must be true of $\kappa$. Hence, when $\kappa$ is
finite at the endpoint of the interval on which
it is finite,\vadjust{\goodbreak} Rockafellar [(\citeyear{MR0274683}), Theorem 24.1]
implies that the derivative extends continuously
to this endpoint, where the derivative at
the endpoint is the one-sided one from within the interval.
Part~\ref{analyticiv} follows directly from this and part \ref
{analytici}.\vspace*{-2pt}
\end{pf}

\section{The irreducible case}\label{primitive}

The discussion starts with a simple lemma which is easily deduced from
Seneta (\citeyear{MR0389944,MR719544}), Theorems 1.1, 1.5.\vspace*{-2pt}
%
%
\begin{lem}\label{matrix}
Let $M$ be an irreducible matrix with all its entries finite and nonnegative.
Then $M$
has a ``Perron--Frobenius'' eigenvalue (which is positive, and of largest
modulus) $e^\rho$,
and there is a finite $C$ that is independent of $n$, $\nu$ and
$\sigma
$ such that
$
e^{-n\rho}(M^n)_{ \nu\sigma} \leq C
$
and, for primitive $M$,
$
n^{-1}\log(M^n)_{ \nu\sigma} \rightarrow\rho$.\vspace*{-2pt}
\end{lem}

In this section it is assumed that there is just one class of types, so
the matrix $m$
is irreducible, that the exponential moment condition~(\ref{mgood})
holds and that $m$ has \PF eigenvalue $\kappa$.
In fact the matrix $m$ is assumed primitive up to the final result in
the section, where periodic $m$ are considered.
Though rather simple, that extension to periodic $m$
is important in establishing the main result.
Most results in this section are not
novel, though several are (I believe) new and their discussion
underpins later developments.
The first lemma is a simple upper bound on transforms that is an
ingredient in the upper bounds on numbers described in the proposition
that follows it.\vspace*{-2pt}
%
%
\begin{lem}\label{transbound}
\[
\limsup_n \frac{1}{n} \log
\biggl(\int e^{ \theta x}Z^{(n)}_{\sigma}(dx)
\biggr) \leq\kappa(\theta) \qquad\mbox{a.s.-}\tP_{ \nu}.
\]
\end{lem}
\begin{pf}Using~(\ref{powm}),
\[
\frac{1}{n} \log\int e^{\theta z} \tE_\nu Z^{(n)}_\sigma(dz)
=
\frac{1}{n} \log(m(\theta)^n)_{ \nu\sigma}.
\]
Lemma~\ref{matrix} implies that
\[
\limsup_n \frac{1}{n} \log
\biggl(\int e^{ \theta x} \tE_\nu Z^{(n)}_{\sigma}(dx)
\biggr) \leq\kappa(\theta) \qquad\mbox{a.s.-}\tP_{ \nu}
\]
and so for any $\varepsilon>0$ and then large enough $n$
\[
\frac{\tE_\nu\int e^{ \theta x} Z^{(n)}_{\sigma}(dx)}{ \exp
(n(\kappa
(\theta)+2\varepsilon))}
\leq
\exp(-n \varepsilon).
\]
This has a finite sum over $n$, giving the result.\vspace*{-2pt}
\end{pf}

The next proposition derives three upper bounds; the first concerns
expectations, the second the probabilities of certain ``extreme'' events
and the third actual numbers. These upper bounds on numbers are (nearly
always) exact: that is the content of Propositions~\ref{behaviour1},
\ref{subcritlower} and~\ref{expgrowthub2}, which are all needed later.\vadjust{\goodbreak}
%
%
\begin{propn}\label{oneclass}
For all $\sigma$, $\nu$, and $a$,
\begin{eqnarray*}
\limsup_n
\frac{1}{n} \log
\bigl(\tE_\nu Z^{(n)}_{\sigma}[na,\infty)
\bigr) &\leq&- \kappa^{\ast}(a),
\\[-2pt]
\limsup_n
\frac{1}{n}\log\bigl(\tP_{ \nu}\bigl(\rB{\sigma}{n}\geq na\bigr)
\bigr)&\leq&
\min\{- \kappa^{\ast}(a),0\}
\end{eqnarray*}
and
\[
\limsup_n \frac{1}{n} \log
\bigl(Z^{(n)}_{\sigma}[na,\infty)
\bigr) \leq- \kappa^{\acci}(a) \qquad\mbox{a.s.-}\tP_{ \nu}.
\]
\end{propn}
\begin{pf}
For $\theta\geq0$,
\[
e^{\theta na} \tE_\nu Z^{(n)}_\sigma[na, \infty)
\leq
\int e^{\theta z} \tE_\nu Z^{(n)}_\sigma(dz)
=
(m(\theta)^n)_{ \nu\sigma}
\]
so that
\[
\log
\bigl(\tE_\nu Z^{(n)}_\sigma[na, \infty)\bigr)
\leq- n \theta a + \log( (m(\theta)^n)_{ \nu
\sigma
}).
\]
Hence, for $\theta\geq0$, using Lemma~\ref{matrix},
\[
\limsup_n \frac{1}{n} \log
\bigl(\tE_\nu Z^{(n)}_\sigma[na, \infty)\bigr) \leq-\bigl(\theta
a-\kappa
(\theta)\bigr).
\]
Since $\kappa$ is defined to be infinite for $\theta<0$, this holds
for all $\theta$ and so
minimizing the right-hand side over $\theta$ gives the first bound. Since
\[
\tP_{ \nu}\bigl(\rB{\sigma}{n}\geq na\bigr)= \tE_\nu I\bigl(\rB{\sigma}{n}
\geq na\bigr)
\leq\tE_\nu Z^{(n)}_{\sigma}[na,\infty),
\]
the second follows directly from this. Turning to the third,
since
\[
e^{\theta na} Z^{(n)}_\sigma[na, \infty)
\leq
\int e^{\theta z} Z^{(n)}_\sigma(dz),
\]
Lemma~\ref{transbound},
gives
\[
\limsup_n \frac{1}{n} \log\bigl( Z^{(n)}_\sigma[na, \infty)\bigr)
\leq-\bigl(\theta a-\kappa(\theta)\bigr)
\qquad\mbox{a.s.-}\tP_{ \nu}
\]
and minimizing over $\theta$ gives the third bound,
with $\kappa^{\ast}$ in place of $\kappa^{\acci}$.
However, $Z^{(n)}_\sigma[na, \infty)$ is integer-valued
and so can only decay geometrically by being zero for all large $n$,
which implies
$\kappa^{\ast}$ can be replaced by $\kappa^{\acci}$.\vspace*{-2pt}
\end{pf}
\begin{pf*}{Proof of Proposition~\ref{behaviour1}}
This is just an application of suitable
large deviation theory based on
\[
\frac{1}{n} \log\int e^{\theta z} \tE_\nu Z^{(n)}_\sigma(dz)
= \frac{1}{n}\log
(m(\theta)^n)_{ \nu\sigma} \rightarrow\kappa(\theta)
\qquad\mbox{for } \theta>0,
\]
which holds by Lemma~\ref{matrix}. See Biggins
[(\citeyear{MR1384364}), Section 7]
for a little more detail on the method.\vspace*{-2pt}
\end{pf*}
%
%
\begin{propn} \label{behaviour1-extra}
\[
\sup_n \frac{1}{n}\log\bigl(\tE_\sigma Z^{(n)}_\sigma[na, \infty)
\bigr) = -\kappa^{\ast}(a).\vadjust{\goodbreak}
\]
\end{propn}
\begin{pf} Note that $a_n=\tE_\sigma Z^{(n)}_\sigma[na, \infty) $
is supermultiplicative ($a_{n+m}\geq a_n a_m$) and so
standard theory of subadditive sequences gives that the supremum
agrees with the limit, and the latter has already been identified in
Proposition~\ref{behaviour1}.
\end{pf}

The next result concerns the decay of the
probability of a particle appearing to the right of $na$.
For the one-type process \citet{MR877384}
gives a result similar to the next one under extra conditions and
Rouault [(\citeyear{MR1198661}), Theorem 2.1] gives a much sharper one.
The multitype case does not seem to have been discussed before.
%
%
\begin{propn}\label{subcritlower} For $a \neq\Up$,
\[
\frac{1}{n}\log\bigl(\tP_{ \nu}\bigl(\rB{\sigma}{n}\geq na\bigr)\bigr)
\rightarrow\min\{-\kappa^{\ast}(a),0\}.
\]
\end{propn}
\begin{pf}
Take $b$ with $b\neq U$ and $\kappa^{\ast}(b)>0$. Take $\varepsilon>0$.
Then, using Propositions~\ref{behaviour1} and~\ref{behaviour1-extra},
there is an $r$ such that
%
%
\begin{equation}\label{niceGW}
-\kappa^{\ast}(b)\geq\frac{1}{r}\log\bigl( \tE_\sigma
Z^{(r)}_{\sigma
}[rb,\infty)\bigr)
\geq-\kappa^{\ast}(b)-\varepsilon.
\end{equation}

Starting from an initial ancestor of type $\sigma$, regard as its
children all its descendants $r$ generations later of type $\sigma$ and
displaced at least $rb$ from
the initial particle's position.
Identify ``children'' of these children in the same way, and so on.
The resulting process is a (one-type) Galton--Watson process with mean
$\tE_\sigma Z^{(r)}_{\sigma}[rb,\infty)$. This process is
subcritical, because
$\exp(-r\kappa^{\ast}(b)) <1$.
Let $N^{(n)}$ be the number in its $n$th generation. Then, by
arrangement, when the initial ancestor is of type $\sigma$,
\[
N^{(n)} \leq Z^{(nr)}_{\sigma}[nrb,\infty)
\]
so that $N^{(n)}>0$ implies that $\rB{\sigma}{nr}\geq nrb$. Hence,
using Asmussen and Hering [(\citeyear{MR701538}), Theorem III.1.6] to
estimate $\tP(N^{(n)}>0)$,
\begin{eqnarray*}
\frac{1}{nr}\log\bigl(
\tP_{ \sigma}\bigl(\rB{\sigma}{nr}\geq nrb\bigr)
\bigr)
&\geq&
\frac{1}{nr}
\log\bigl(\tP\bigl(N^{(n)}>0\bigr)\bigr)
\\
&\rightarrow&
\frac{1}{r}
\log\bigl(\tE_\sigma Z^{(r)}_{\sigma}[rb,\infty)\bigr)
\\
&\geq&
-\kappa^{\ast}(b)-\varepsilon.
\end{eqnarray*}
Now, consider a process started from a type $\nu$. Because $m$ is
primitive, there is an $s$ such that $m^n$ has all entries strictly
positive for every $n \geq s$. Then, for a suitable~$T$, there is a
positive probability of a descendant in generation $s+r'$ of type
$\sigma$ and to the right of $T$ for each of $r'=0,1,2,\ldots,r-1$. Let
$p$ be the minimum\vadjust{\goodbreak} of these probabilities. For $b>a$, all sufficiently
large $n$ and $r'=0,1,2,\ldots,r-1$,
\begin{eqnarray*}
\tP_{ \nu}\bigl(\rB{\sigma}{nr+s+r'}\geq(nr+s+r')a\bigr)
&\geq&
\tP_{ \nu}\bigl(\rB{\sigma}{nr+s+r'}\geq nrb+T\bigr)
\\
&\geq&
p \tP_{ \sigma} \bigl(\rB{\sigma}{nr}\geq nrb\bigr).
\end{eqnarray*}
Therefore
\begin{eqnarray*}
\liminf_n \frac{1}{n}\log\bigl( \tP_{ \nu}
\bigl(\rB{\sigma}{n}\geq na\bigr)\bigr)
&\geq&
\liminf_n \frac{1}{nr}\log\bigl(\tP_{ \sigma} \bigl(\rB{\sigma
}{nr}\geq
nrb\bigr)\bigr)
\\
&\geq&
-\kappa^{\ast}(b)-\varepsilon.
\end{eqnarray*}
This holds for any $\varepsilon>0$ and $b>a$. Thus, since $\kappa^{\ast
}$ is
continuous from the right
except at $\Up$,
\[
\liminf_n \frac{1}{n}\log\bigl( \tP_{ \nu}
\bigl(\rB{\sigma}{n}\geq na\bigr)
\bigr) \geq\min\{-\kappa^{\ast}(a),0\}
\]
except possibly for $a=\Up$.
The upper bound in Proposition~\ref{oneclass} completes the proof.
\end{pf}
%
%
\begin{lem}\label{r-funct}
Suppose that the branching process is supercritical [i.e., $\kappa(0)>0$].
Then $\kappa^{\acci}$ is an $r$-function (as introduced at Definition
\ref
{defr-function}).
\end{lem}
\begin{pf}
Lemma~\ref{analytic} gives that $\kappa$ is $k$-convex and closed.
Also, $\kappa(0)>0$ because the process is supercritical. Hence, using
Lemma~\ref{covex-d}, $\kappa^{\ast}$ is increasing, less than zero
somewhere, and convex. Thus $\kappa^{\acci}$ is a proper convex
function that is strictly negative somewhere, left-continuous and
infinite when strictly positive and so is an $r$-function.
\end{pf}
\begin{pf*}{Proof of Proposition~\ref{expgrowthub2}}
The argument is very similar to that for Proposition
\ref{subcritlower}. It will be convenient to let $\Surv$ be the
survival set of the process, even though $\tP_{ \nu}(\Surv)=1$.
Proposition~\ref{oneclass} implies that~(\ref{expgrowthub2}) holds for
$a>\Gamma(\kappa^{\ast})$, with the limit being $-\infty$. Hence, only
$a< \Gamma(\kappa^{\ast})$ need to be considered. Take $b>a$ but with
$\kappa^{\ast}(b)<0$, which is possible because, by Lemma~\ref{r-funct},
$\kappa^{\acci}$ is an $r$-function, and take
$\varepsilon\in(0,-\kappa^{\ast}(b))$. As in Proposition~\ref{subcritlower},
use Propositions~\ref{behaviour1} and~\ref{behaviour1-extra}, to choose
$r$ such that~(\ref{niceGW}) holds. Start from an initial ancestor of
type $\sigma$, and identify the embedded (one-type) Galton--Watson
process as in Proposition~\ref{subcritlower}. This now has mean
$\tE_\sigma Z^{(r)}_{\sigma}[rb,\infty)$ and is supercritical, because
$\exp(-r(\kappa^{\ast}(b)+\varepsilon)) >1$. Let $N^{(n)}$ be the
number in
its $n$th generation. Then, using, for example, Asmussen and Hering
[(\citeyear{MR701538}), Theorems II.5.1, II.5.6] to get the limit of
$n^{-1}\log N^{(n)}$,
\begin{eqnarray*}
\frac{1}{nr}\log\bigl( Z^{(nr)}_{\sigma}[nrb,\infty) \bigr)
&\geq&
\frac{1}{nr}
\log N^{(n)}
\\
&\rightarrow&
\frac{1}{r}\log\bigl(\tE_\sigma Z^{(r)}_{\sigma}[rb,\infty)\bigr)
\\
&\geq&
-\kappa^{\ast}(b)-\varepsilon
\end{eqnarray*}
on the survival set of $N^{(n)}$, which has positive probability. Three
matters remain: allowing initial types different from $\sigma$; dealing
with generations that are not a multiple
of $r$; and showing the result holds almost surely on the survival set
of the whole process and not just that of some embedded one.
The argument for dealing with all three is standard,
and the idea is not complicated. It is to run the process to some large
generation, allow each type $\sigma$ then present to initiate its own
$N^{(n)}$, and then use
any that survives to provide a suitable lower bound. Here is a more
careful version.

Fix $\sigma$. Let $\{z_i^{(s)}\dvtx i\}$ be the points of
$Z^{(s)}_{\sigma
}$. Recall that $\cF{s}$ contains all information on families with the
parent in a generation up to and including $s-1$. Let
$N^{(n)}_{s,i}$ be the process $N^{(n)}$
initiated by the particle at $z_i^{(s)}$.
By arrangement, $N^{(n)}_{s,i}$ contains points
in the $(nr+s)$th generation to the right of $nrb+z_i^{(s)}$.
Given $\cF{s}$, these processes are independent.
Let $\Surv(s)$ be the
event that at least one of these processes survives.
Fix $s$ and $r'$.
For any $i$, for all large enough $n$,
$(nr+sr+r')a-z_i^{(sr+r')}\leq nrb$ and so
\[
Z^{(nr+sr+r')}_{\sigma}\bigl[(nr+sr+r')a,\infty\bigr)
\geq N^{(n)}_{(sr+r'),i}
\]
for all sufficiently large $n$. Hence
%
%
\begin{equation}\label{skeleton}
\liminf_{n} \frac{1}{(nr+r')}\log
\bigl(Z^{(nr+r')}_{\sigma}\bigl[(nr+r')a,\infty\bigr)\bigr)
\geq-\kappa^{\ast}(b)-\varepsilon
\end{equation}
on $\Surv(sr+r')$. Furthermore $\Surv(sr+r')\subset\Surv
((s+1)r+r')\subset\Surv$ and $\tP_{ \nu}(\Surv(sr+r'))\uparrow
\tP_{
\nu}(\Surv)$
as $r\uparrow\infty$. Hence~(\ref{skeleton}) holds almost surely on
$\Surv$ for each $r'=0,1,2, \ldots, r-1$. Also, it
holds for any $\varepsilon>0$ and every $b>a$. Since $\kappa^{\ast}$ is
continuous from the right at $a$,
this provides the lower bound to complement the upper bound in
Proposition~\ref{oneclass}.

Though it does not matter here,
it is perhaps worth noting that, because $Z^{(n)}_\sigma[na, \infty)$
is monotone in
$a$, the null set in~(\ref{expgrowthub2}) can be taken independent
of~$a$.
\end{pf*}

Since the proof of Theorem~\ref{altermaintheorem} will be by
induction on $K$ it is worth stating explicitly
that the induction starts successfully.
%
%
\begin{cor}\label{K=1}
When $K=1$, Theorem~\ref{altermaintheorem} holds.
\end{cor}
\begin{pf} For $K=1$, the condition~(\ref{phiiorder}) is
equivalent to~(\ref{mgood}) and the conditions~(\ref{newphiiorder}),
(\ref{eitheror1}) and~(\ref{ii}) are vacuous.
Proposition~\ref{supercrit} now gives the required conclusions.
\end{pf}

When $m$ is irreducible with period $d>1$, $m^d$ has $d$
primitive blocks on its diagonal, each with \PF eigenvalue $\kappa^d$.
These primitive blocks partition the types into $d$ subclasses.\vadjust{\goodbreak} The
next result deals with the case where $\nu$ and~$\sigma$ are in the
same subclass.
It is possible to say a bit more, dealing with $\nu$ and~$\sigma$
in different subclasses, but this is
not needed here.\vspace*{-2pt}
%
%
\begin{propn}\label{irredcase} If ``primitive'' is replaced by
``irreducible with period \mbox{$d>1$},'' then Propositions~\ref{behaviour1}
and~\ref{supercrit} and
all the results in this section continue to hold, provided ``$n $'' is
replaced by
``$nd $'' and $\nu$ and $\sigma$ come from the same subclass.\vspace*{-2pt}
\end{propn}
\begin{pf} Apply the results to the primitive process
obtained by only inspecting every $d$th generation.\vspace*{-2pt}
\end{pf}

\section{Lower bounds on numbers, main results}\label{lower}

The objective in this section is to prove Theorem~\ref{prelimmaintheorem}.
The main challenge is to show how in a sequential process the numbers
in the penultimate class contribute to
numbers in the final class. The first proposition shows two things:
that the numbers in the penultimate
class drive the numbers of those first in their line of descent to be
in the final class and
that those numbers drive the first in the line of descent of any other
type in the final class. To discuss this,
let $F^{(n)}_{\sigma}$ be the point process of those in generation $n$
of type $\sigma$ that are first in their line of descent with this type.
The subsequent theorem explores how the numbers in $F^{(n)}_{\sigma}$
combine with the
growth of numbers within the class.\vspace*{-2pt}
%
%
\begin{propn}\label{inductthm-a1}
Consider a sequential process.
Let $\upsilon\in\pen$ and
$\tau\in\last$ be types for which
$m_{\upsilon\tau}> 0$ and let $\nu\in\first$. If there is an
$r$-function $r$ such that for all $a < \Gamma(r)$
\[
\liminf\frac{1}{n} \log
\bigl(Z^{(n)}_{\upsilon}[na,\infty)
\bigr) \geq-r(a) \qquad\mbox{a.s.-}\tP_{ \nu},
\]
then
%
%
\begin{equation}\label{firstdriver}
\liminf_n \frac{1}{n} \log
\bigl(F^{(n)}_{\sigma}[na,\infty)
\bigr) \geq-r(a) \qquad\mbox{a.s.-}\tP_{ \nu}
\end{equation}
for all
$a \neq\Gamma(r)$ and $\sigma\in\last$.\vspace*{-2pt}
\end{propn}
%
%
\begin{theorem}\label{inductthm}
Consider any process with
final class $\last$ having \PF eigen\-value $\kappa$ and initial type
$\nu\notin\last$.
Suppose that for the $r$-function $r$ and any $\sigma\in\last$,
(\ref{firstdriver}) holds
for all
$a < \Gamma(r)$.
Then
\[
\liminf_n \frac{1}{n} \log
\bigl(Z^{(n)}_{\sigma}[na,\infty)
\bigr) \geq- {\mathfrak C}[r,\kappa^{\ast}]^{\circ}(a)
\qquad\mbox{a.s.-}\tP
_{ \nu}
\]
for all
$a < \Gamma({\mathfrak C}[r,\kappa^{\ast}])$.\vspace*{-2pt}
\end{theorem}

Before starting the main proofs, three lemmas are proved. The second of these
identifies a characterization of ${\mathfrak C}[r,\kappa^{\ast}]$
that arises
in proving
Theorem~\ref{inductthm}.\vspace*{-2pt}
%
%
\begin{lem} \label{s-igood} Suppose $f$ is $k$-convex,
$r$ is an $r$-function and ${\mathfrak M}[r^{\ast},f](\phi)<\infty$
for some
$\phi
>0$. Then ${\mathfrak C}[r,f^{\ast}]^{\circ}$ is also an $r$-function.\vadjust{\goodbreak}
\end{lem}
\begin{pf}
By Lemma~\ref{covex-d}, $f^{\ast}$ is proper, closed, convex and increasing.
Clearly ${\mathfrak C}[r,f^{\ast}]^{\circ}$
is convex. It is increasing, because both $r$ and
$f^{\ast}$ are, and negative somewhere,
because $r$ is. Since ${\mathfrak C}[r,f^{\ast}]$ is continuous from
the left
(by definition) the same must be true of ${\mathfrak C}[r,f^{\ast
}]^{\circ}$.
Finally, using both parts of Lemma~\ref{kh}, $({\mathfrak M}[r^{\ast
},f])^{\ast}={\mathfrak C}[r,f^{\ast}]$, and now Lemma \ref
{covex-d}\ref{covex-di} implies that
$({\mathfrak M}[r^{\ast},f])^{\ast}$ is not identically $-\infty$.
\end{pf}
%
%
\begin{lem}\label{concmaj}
$\!\!\!$Under the same conditions as Lemma~\ref{s-igood}, for $a<\Gamma
({\mathfrak C}[r,f^{\ast}])$,
\[
{\mathfrak C}[r,f^{\ast}](a)=\inf\{\lambda
r(b)+(1-\lambda)f^{\ast}(c)\dvtx
(\lambda
,b,c) \in A_a, r(b)<0\},
\]
where
$A_a=\{(\lambda,b,c)\dvtx\lambda\in[0,1], \lambda b+ (1-\lambda) c =a,
\lambda r(b)+(1-\lambda)f^{\ast}(c)<0\}$.
\end{lem}
\begin{pf}
Let ${\mathfrak c}[f,g]$ be the convex minorant of $f$ and $g$, so that
${\mathfrak C}[f,g]$
is the closure of ${\mathfrak c}[f,g]$. Since ${\mathfrak C}[r,f^{\ast}]$
is increasing and convex, it is continuous and strictly negative on
$(-\infty,\Gamma({\mathfrak C}[r,f^{\ast}]))$
and so on that set
${\mathfrak C}[r,f^{\ast}](a)={\mathfrak c}[r,f^{\ast}](a)$.
Furthermore, using Rockafellar
[(\citeyear{MR0274683}), Theorem 5.6],
\[
{\mathfrak c}[r,f^{\ast}](a)=\inf\{\lambda
r(b)+(1-\lambda)f^{\ast}(c)\dvtx\lambda
\in[0,1],
\lambda b+ (1-\lambda) c =a\},
\]
which equals
$\inf\{\lambda r(b)+(1-\lambda)f^{\ast}(c)\dvtx(\lambda,b,c) \in A_a\}$
when ${\mathfrak c}[r,f^{\ast}](a)<0$.
It remains to show that the additional constraint $r(b)<0$ makes no
difference, by showing that excluded values of the function can be
approximated closely by included ones.
The only possibility excluded is
$b=\Gamma(r)$, since
$r$ is infinity when strictly positive. The corresponding values of the
function being minimized can be approximated arbitrarily well when
$\lambda<1$ by taking $b \uparrow\Gamma(r)$ keeping $c$ fixed and
adjusting $\lambda$.
To deal with the $\lambda=1$ case, where $a=b=\Gamma(r)$,
note first that if $f^{\ast}(\tilde{a})=\infty$
for all $\tilde{a}>\Gamma(r)$, then, because
$r(\tilde{a})=\infty$ for all $\tilde{a}>\Gamma(r)$
also,
the same will be true of the convex minorant of
$r$ and $f^{\ast}$. Then
$a=\Gamma(r)=\Gamma({\mathfrak C}[r,f^{\ast}])$,
contradicting $a<\Gamma({\mathfrak C}[r,\kappa^{\ast}])$.
Hence,
there must be a $c>a$ with $f^{\ast}(c)<\infty$.
Then
\[
(1-\varepsilon) r\biggl(\frac{a-\varepsilon c}{1-\varepsilon}
\biggr)+\varepsilon f^{\ast}(c)
\]
provides a suitable approximation as $\varepsilon\downarrow0$.
\end{pf}
%
%
\begin{lem}\label{bin-ld2}
Let $Y_n$ be Binomial on $N_n $ trials with success
probability $p_n$ and
$\sum_n (N_n p_n)^{-1}(1-p_n)<\infty$.
Then
$
\log(Y_n)-\log(N_n p_n) \rightarrow0$ as
$n \rightarrow\infty$
almost surely.
\end{lem}
\begin{pf}
$\!\!\!$Chebyshev's inequality gives
that $P(|Y_n-EY_n|\geq\varepsilon EY_n )$ is bound\-ed above by
$(\varepsilon^2 N_n p_n)^{-1}(1-p_n)$,
and so Borel--Cantelli gives that
$Y_n/\break(N_np_n) \rightarrow1$.
\end{pf}
\begin{pf*}{Proof of Proposition~\ref{inductthm-a1}} Since
$r(a)=\infty$ for $a>\Gamma(r)$, the result holds in these cases.
Assume now that $a<\Gamma(r)$. The result is proved first for
$\sigma=\tau$. For some $T$ there is a probability $p>0$ that
a
particle of type $\upsilon$ has a child\vadjust{\goodbreak} of type $\tau$ to the right of
$T$, because $m_{\upsilon\tau}>0$. Then, given $\cF{n}$,
$\ntF^{(n+1)}[nb-T,\infty)$ is bounded below by a Binomial variable,
$Y_{n}$, on $Z^{(n)}_{\upsilon}[nb,\infty)$ trials with success
probability~$p$. Take $b \in(a,\Gamma(r))$ with $r(b)<0$. Then, by
Lemma~\ref{bin-ld2}, for $\varepsilon>0$ and then large enough $n$
\[
\log\bigl(\ntF^{(n+1)}[nb-T,\infty) \bigr)\geq\log(Y_{n}) \geq
\log\bigl(pZ^{(n)}_{\upsilon}[nb,\infty)\bigr) -\varepsilon.
\]
Hence
\[
\liminf\frac{1}{n}\log\bigl(\ntF^{(n+1)}[nb-T,\infty)\bigr)
\geq-r(b)
\]
and so
\[
\liminf\frac{1}{n}\log\ntF^{(n)}[na,\infty)\geq-r(b) \uparrow-r(a)
\]
as $ b \downarrow a$, giving~(\ref{firstdriver}) for $a<\Gamma(r)$
when $\sigma=\tau$.

Suppose now that $\sigma\neq\tau$. Find a sequence of distinct types
$ \tau=\sigma(0)\neq\sigma(1)\neq\cdots\neq\sigma
(c)=\sigma$ such
that each type can have children of the type
following it in the sequence. For some $T$,
there is a probability $p>0$ that a particle of type $\tau$ has a descendant
$c$ generations later to the
right of $T$ and of type $\sigma$. Let
$\tF^{(n+c)}$ be the point process
of all those in $F^{(n+c)}_{\sigma}$
with ancestors of type $\tau$ in generation $n$.
Then, given $\cF{n}$,
$\tF^{(n+c)}[nb-T,\infty)$ is bounded below by a Binomial variable, $Y_{n}$,
on $F^{(n)}_{\tau}[nb,\infty)$ trials with success probability $p$. Thus
\[
\liminf\frac{1}{n}\log\tF^{(n)}[na,\infty)\geq-r(a),
\]
when $r(a)<0$. Clearly $F^{(n)}_{\sigma}[x,\infty)\geq\tF
^{(n)}[x,\infty)$, giving the result.
\end{pf*}
\begin{pf*}{Proof of Theorem~\ref{inductthm}}
Let $d$ be the period of $\last$. Take $b< \Gamma(r)$ with $r(b)<0$,
$c< \Gamma(\kappa^{\ast})$ with $\kappa^{\ast}(c)<0$, $\varepsilon>0$
and $\lambda\in[0,1]$. For each positive integer~$t$, let $n=n(t)$ and
$\tilde{n}=\tilde{n}(t)$ be chosen to be increasing in $t$ with
$t=n+\tilde{n}d$ and with $n/t \rightarrow\lambda$ as $n
\rightarrow\infty$. Let $N_t=F^{(n)}_{\sigma}[nb,\infty)$. Then, using
the assumption that~(\ref{firstdriver}) holds, provided $n=n(t)
\rightarrow\infty$,
\begin{eqnarray*}
\liminf_t \frac{1}{t}\log N_t
&=&
\liminf_t \frac{1}{t}\log
\bigl( F^{(n)}_\sigma[nb,\infty)\bigr) \\
&=& \lambda\liminf_n \frac{1}{n}\log
\bigl( F^{(n)}_\sigma[nb,\infty)\bigr)\\
&\geq& -\lambda r(b).
\end{eqnarray*}
Given $\cF{n}$, $Z^{(t)}_{\sigma}[nb+\tilde{n}dc,\infty)$
is bounded below by $N_t$
independent copies (under~$\tP_{ \sigma}$) of $Z^{(\tilde
{n}d)}_{\sigma
}[\tilde{n}dc,\infty)$.
Propositions~\ref{behaviour1},~\ref{supercrit} and~\ref{irredcase}
imply that most of these copies should
have size near $\exp(- \tilde{n}d\kappa^{\ast}(c))$.
Let
$Y_t$ be the number that are not too far below their expectation, that
is, the number with
\[
\log\bigl(Z^{(\tilde{n}d)}_{\sigma}[\tilde{n}dc,\infty)\bigr)
\geq\tilde{n}d\bigl(-\kappa^{\ast}(c)-\varepsilon\bigr).
\]
Then, given $\cF{n}$, $Y_t$ is a Binomial variable with $N_t$ trials
and success probability~$p_{t}$,
where
\[
p_{t}=
\tP_{ \sigma} \bigl(
\log\bigl(Z^{(\tilde{n}d)}_{\sigma}[\tilde{n}dc,\infty)\bigr)
\geq\tilde{n}d\bigl(-\kappa^{\ast}(c)-\varepsilon\bigr)
\bigr).
\]
Propositions~\ref{supercrit} and~\ref{irredcase}
imply that $p_{t} \rightarrow1$ provided
$\tilde{n}(t) \rightarrow\infty$.
Now
\[
\log\bigl(Z^{(t)}_{\sigma}[nb+\tilde{n}dc,\infty)\bigr) \geq
\log
Y_t+\tilde{n}d\bigl(-\kappa^{\ast}(c)-\varepsilon\bigr)
\]
and, using Lemma~\ref{bin-ld2},
$Y_t/N_t \rightarrow1$ almost surely when $\sum_t (1/N_t)<\infty$. Let
$T(j)=\max\{t\dvtx n(t)=j\}$.
For suitable small $\delta$ and then all sufficiently large~$n$
\[
\log N_t = \log\bigl( F^{(n)}_\sigma[nb,\infty)\bigr) \geq n
\bigl(-r(b)-\delta\bigr)>0.
\]
Then,
\[
\sum_t \frac{1}{N_t}
\leq C\sum_j \frac{T(j)}{\exp(j (-r(b)-\delta))}
\]
and this is finite
provided $T$ does not grow exponentially quickly, for which it suffices that
$n(t)^\gamma\geq t$ for some $\gamma>1$.
Putting this together, provided $\tilde{n}(t) \rightarrow\infty$ and
$n(t)^\gamma\geq t$,
which can both be arranged,
%
%
\begin{equation}\label{subseq}\quad
\liminf_t \frac{1}{t}\log\bigl(Z^{(t)}_{\sigma}[nb+\tilde
{n}dc,\infty)
\bigr)\geq
\lambda(-r(b))+(1-\lambda)\bigl(-\kappa^{\ast}(c)-\varepsilon\bigr).
\end{equation}
Note too that
\[
\frac{nb+\tilde{n}dc}{t}= \biggl( \frac{n}{t}b+ \frac{\tilde{n}d}{t}c
\biggr) \rightarrow
\lambda b+(1-\lambda)c
\]
so that~(\ref{subseq}) implies, using continuity of $r$ at $b$ and
$\kappa^{\ast}$ at $c$,
%
%
\begin{equation}\label{lowerbd}\qquad
\liminf_t\frac{1}{t}\log
\bigl(Z^{(t)}_{\sigma}\bigl(t \bigl[\lambda b+(1-\lambda)c\bigr),\infty\bigr) \bigr)
\geq-\bigl(
\lambda r(b)
+
(1-\lambda)\kappa^{\ast}(c)\bigr).
\end{equation}

Consider instead the case where $\kappa^{\ast}(c) \geq0$, but still with
$t=n(t)+\tilde{n}(t)d$. Let
$p_{t}=
\tP_{ \sigma}(\rB{\sigma}{\tilde{n}d} \geq\tilde
{n}dc)$.
Now,
given $\cF{n}$, $Z^{(t)}_{\sigma}[nb+\tilde{n}dc,\infty)$
is bounded below by a Binomial variable, $Y_{t}$, on
$N_t=F^{(n)}_{\sigma}[nb,\infty)$
trials with success probability
$p_{t}$.
Much as previously,
provided $n(t) \rightarrow\infty$,
$\tilde{n}(t) \rightarrow\infty$ and
$n(t)/t \rightarrow\lambda$,
as $t \rightarrow\infty$,
Propositions~\ref{subcritlower} and~\ref{irredcase} give
\[
\liminf\frac{1}{t}(\log N_t + \log p_{t} )\geq-\bigl( \lambda r(b) +
(1-\lambda)\kappa^{\ast}(c)\bigr).
\]
Therefore, using Lemma~\ref{bin-ld2},
when $
\lambda r(b)
+
(1-\lambda)\kappa^{\ast}(c)<0$,
\begin{eqnarray*}
\liminf\frac{1}{t} \log
\bigl(Z^{(t)}_{\sigma}[nb+\tilde{n}dc,\infty)
\bigr)
&\geq&
\liminf\frac{1}{t} \log Y_t
\\
&\geq&
-\bigl(\lambda r(b)
+
(1-\lambda)\kappa^{\ast}(c)\bigr)
\end{eqnarray*}
and so, using continuity of $r$ at $b$,
(\ref{lowerbd}) holds in this case, too.\vadjust{\goodbreak}

Hence~(\ref{lowerbd}) holds for any $\lambda\in[0,1]$, any $b$ such that
$r(b)<0$ and any~$c$ with
$\lambda r(b)
+
(1-\lambda)\kappa^{\ast}(c)<0$. Fix $a$. Maximize the
right of~(\ref{lowerbd}),
using Lemma~\ref{concmaj}, over $(\lambda,b,c) \in A_a$ with $r(b)<0$
to get
\[
\liminf_t \frac{1}{t}\log\bigl(Z^{(t)}_{\sigma}[ta,\infty)
\bigr)\geq
{\mathfrak C}[r,\kappa^{\ast}](a).
\]
Now use that $Z^{(t)}_{\sigma}[ta,\infty)$ is integer-valued to replace
${\mathfrak C}[r,\kappa^{\ast}]$ by ${\mathfrak C}[r,\kappa^{\ast
}]^{\circ}$.
\end{pf*}
\begin{pf*}{Proof of Theorem~\ref{prelimmaintheorem}}
The result holds for $K=1$, by Corollary~\ref{K=1}. Suppose the result
holds for $K - 1$. By Lemmas~\ref{analytic} and~\ref{s-igood}, $r_K$
has the right properties. Then, by Proposition~\ref{inductthm-a1} and
then Theorem~\ref{inductthm},~(\ref{keyresultlb}) holds.
\end{pf*}

\section{Properties of $\nat{f}$ and the recursion}\label{furtherconv}

The main objectives of this section are to prove Proposition
\ref{f-natural2} giving properties of $\nat{f}$ and to establish
Proposition~\ref{alternativerec} giving the alternative recursion for
$r_i$.

Recall that $\nat{f}$ is the maximal convex function that has $\nat
{f}(\theta)/\theta$
monotone decreasing in $\theta\in(0,\infty)$
such that $\nat{f}\leq f$, and that $\varthetapp{f}$
is given by~(\ref{defnvartheta}).
The next result describes the structure of $\nat{f}$
and shows $\varthetapp{f}$
is closely connected to $\Gamma(f^{\ast})$.
It is worth mentioning that, although this proposition admits other
possibilities,
in the main results here $\cl{f}(\vartheta)$ and $f(\vartheta)$ will
only be different in cases where $f(\vartheta)$ is also infinite.
The formula
$\Gamma(f^{\ast})=\inf\{f(\theta)/\theta\dvtx\theta>0\}$
included in the proposition is the one used for the speed in the
irreducible blocks by
\citet{MR2322849} in their model.
%
%
\begin{propn}\label{f-natural2}\label{sw-nat}
Suppose $f$ is $k$-convex. Let $\Gamma= \Gamma(f^{\ast})$,
$\vartheta
=\varthetapp{f}$ and $\psiu=\inf\PhiD( f)$.
Then $\nat{f}\equiv-\infty$ and $\vartheta=-\infty$ when $\Gamma
=-\infty$. Otherwise, $\vartheta\geq0$ and $\nat{f}(\theta
)=f(\theta
)\mbox{ for } 0 \leq\theta< \vartheta$ (by definition). When $0
\leq
\vartheta<\infty$,
\[
\nat{f}(\theta)=\theta\Gamma<f(\theta)\qquad\mbox{for } \theta>
\vartheta
\]
and
\[
\nat{f}(\vartheta)=\cases{
f(\vartheta)\geq\cl{f}(\vartheta)=\vartheta\Gamma,
&\quad when $\vartheta=\psiu$,\vspace*{1pt}\cr
\vartheta\Gamma=\cl{f}(\vartheta)\leq f(\vartheta), &\quad when
$\vartheta>\psiu$.}
\]
In all cases,
%
%
\begin{equation}\label{w-speed}
\Gamma=\inf_{\theta>0} \frac{\nat{f}(\theta)}{\theta}=\inf
_{\theta>0}
\frac{f(\theta)}{\theta}.
\end{equation}
When $0 \leq\vartheta<\infty$,
$
\Gamma=f(\varthetapps{f})/\varthetapps{f}
$
provided $f$ is lower semi-continuous at $\varthetapps{f}$ and, when
$\varthetapps{f}=\infty$, $
\Gamma=\lim_{\theta\uparrow\infty} f(\theta)/\theta
$.
\end{propn}

Recall that $f^{\acci}$ is defined to be $(f^{\ast})^{\circ}$. Let
\[
f^{\flat}=(f^{\acci})^{\ast}=((f^{\ast})^{\circ})^{\ast}
\]
and
\[
\vartheta^{\flat}(f)=\inf\{\theta\dvtx f^{\flat}(\theta)<\cl
{f}(\theta)\},\vadjust{\goodbreak}
\]
which is $+\infty$ when this set is empty. Let $\psiu=\inf\PhiD( f)$.
The next lemma, which will be proved later in the section, says that
$\nat{f}$ and $f^{\flat}$ can only be different at $\psiu$ where the
former is $f(\psiu)$ and the latter is $\cl{f}(\psiu)$. This motivates
deriving properties of $f^{\flat}$.
%
%
\begin{lem}\label{nat=flt}
Let $f$ be $k$-convex. Then $\varthetapp{f}=\vartheta^{\flat}(f)$. When
$\varthetapp{f}= -\infty$,
$\nat{f}=f^{\flat}\equiv-\infty$. When $\varthetapp{f}\geq0$,
$\nat
{f}(\theta)=f^{\flat}(\theta)$ for $\theta>\psiu$, and
$
\nat{f}(\psiu)=f(\psiu)\geq\cl{f}(\psiu)=f^{\flat}(\psiu)
$.
\end{lem}

The next result establishes some properties of
$f^{\flat}$. In particular, the second part shows that it is a candidate
for $\nat{f}$, in that it has the right properties.
Building on these properties, the result following this one
characterizes $f^{\flat}$.
%
%
\begin{lem}\label{f-natural1} Let $f$ be $k$-convex and $\Gamma=
\Gamma(f^{\ast})$.
{\renewcommand\thelonglist{(\roman{longlist})}
\renewcommand\labellonglist{\thelonglist}
\begin{longlist}
\item\label{f-natural1-i}
$f^{\flat}(\theta)=\sup_{a \leq\Gamma} \{ \theta a - f^{\ast}(a)\}
$ when
$\Gamma>-\infty$, and $f^{\flat}\equiv-\infty$ when $\Gamma
=-\infty$;
\item\label{f-natural1-ii} $f^{\flat}\leq f$ and $f^{\flat}(\theta
)/\theta$
is decreasing as $\theta$ increases, so $f^{\flat} \leq\nat{f}$;
\item\label{f-natural1-iii} When $\theta'\geq\theta$, $f^{\flat
}(\theta
')\leq f^{\flat}(\theta)+(\theta'-\theta)\Gamma$.
\end{longlist}}
\end{lem}
\begin{pf}
Since $f^{\ast}(a)>0$ for $a>\Gamma$ and these are swept to infinity
in $f^{\acci}$, applying the definitions gives~\ref{f-natural1-i}.
Now
\[
f^{\flat}(\theta)=\sup_{a \leq\Gamma}
\{\theta a - f^{\ast}(a)\}\leq
\sup_a \{\theta a - f^{\ast}(a)\}=
\cl{f}(\theta)\leq f(\theta)
\]
using Lemma~\ref{kh} for the second equality.
Also,
\[
\frac{f^{\flat}(\theta)}{\theta}= \sup_{a \leq\Gamma} \biggl\{ a -
\frac
{f^{\ast}(a)}{\theta}
\biggr\}
\]
and $f^{\ast}(a)\leq0$ for these $a$,
so this decreases as $\theta$ increases. This proves~\ref{f-natural1-ii}.
Maximizing
$\theta' a - f^{\ast}(a)=\theta a - f^{\ast}(a)+(\theta'-\theta)a$
over $a \leq\Gamma$ completes the proof.
\end{pf}

At this point an additional convexity idea is needed.
The subdifferential at $\phi$ of a convex $f$,
$\subdiff f (\phi)$, is defined as the set of slopes
of possible tangents to $f$ at $\phi$. More formally,
\[
\subdiff f (\phi)=\{a\dvtx f(\theta)\geq f(\phi)+ a (\theta-\phi
)\mbox{ }\forall
\theta\}.
\]
The set is empty when $f$ is infinite at $\phi$ or has a one-sided derivative
at $\phi$ that is infinite in modulus,
it contains a single value at points where $f$ is differentiable,
and it is a
nondegenerate closed interval in all other cases; see
\citet{MR0274683}, Theorems 23.3, 23.4.
In the last case the infimum of $\subdiff f(\phi) $ is the left point of
this interval and is the derivative of $f$ from the left there.
%
%
\begin{lem}\label{subdiffeasy} Suppose $f$ is proper and convex.
{\renewcommand\thelonglist{(\roman{longlist})}
\renewcommand\labellonglist{\thelonglist}
\begin{longlist}
\item\label{subdiffeasyi}If $f$ is finite in a neighborhood of
$\phi$,
then $\subdiff f(\phi)=\subdiff\cl{f}(\phi)$ and is certainly nonempty.
\item\label{subdiffeasyii}
The following are equivalent:
$\gamma\in\subdiff f(\phi)$; $\phi\gamma-f(\phi)=f^{\ast}(\gamma)\hspace*{-0.9pt}$
$(\mbox{$=$}\sup\{\theta\gamma-f(\theta)\dvtx\theta\})$.
\item\label{subdiffeasyiii}If $f(\phi)=\cl{f}(\phi)$, the statements
in~\ref{subdiffeasyii} are also equivalent to
$\phi\in\subdiff f^{\ast}(\gamma)$ and to
$\phi\gamma-f^{\ast}(\gamma)=\sup\{a \phi-f^{\ast}(a)\dvtx a\}$ $(\mbox{$=$}
f(\phi
))$.
\end{longlist}}
\end{lem}
\begin{pf}The assertion that $\subdiff f(\phi)$ is nonempty is in
\citet{MR0274683}, Theorem 23.4.
The equivalences
are some of the results in \citet{MR0274683}, Theorem~23.5.
\end{pf}
%
%
\begin{lem} \label{rateagreenew}
Let $h$ be $k$-convex with $h(\phi)<\infty$.
Suppose $g$ is convex, $g \geq h$, $g(\phi)=h(\phi)$
and $\gamma\in\subdiff h (\phi)$.
Then:
{\renewcommand\thelonglist{(\roman{longlist})}
\renewcommand\labellonglist{\thelonglist}
\begin{longlist}
\item\label{rateagreenewi}$\gamma\in\subdiff g (\phi)$ and
$g^{\ast}(\gamma)=h^{\ast}(\gamma)$;
\item\label{rateagreenewii}if $h(\theta)=g(\theta)$ for all
$\theta
\leq\phi$, then
$g^{\ast}(a)=h^{\ast}(a)$ for all $a \leq\gamma$;
\item\label{rateagreenewiii}if, in addition, $g(\theta)=\infty$ for
$\theta>\phi$, then $g^{\ast}(a)=h^{\ast}(\gamma)-\phi(\gamma
-a)=\phi a -
h(\phi)$ for $a>\gamma$.
\end{longlist}}
\end{lem}
\begin{pf} Since $g(\phi)=h(\phi)$ and $g \geq h$,
\begin{eqnarray*}
\subdiff h (\phi)
&=&
\{a\dvtx h(\theta)\geq h(\phi)+ a (\phi-\theta)\mbox{ }\forall\theta\}\\
&\subset&
\{a\dvtx g(\theta)\geq g(\phi)+ a (\phi-\theta)\mbox{
}\forall\theta\}\\
&=& \subdiff g (\phi).
\end{eqnarray*}
Thus
$\gamma\in\subdiff h(\phi)$
implies $\gamma\in\subdiff g(\phi)$, and
then Lemma~\ref{subdiffeasy}\ref{subdiffeasyii} gives
\[
h^{\ast}(\gamma)=\sup_\theta\{\theta\gamma-h(\theta) \}= \phi
\gamma
-h(\phi)=\phi\gamma-g(\phi)
=\sup_\theta\{\theta\gamma-g(\theta) \}=g^{\ast}(\gamma).
\]
This proves~\ref{rateagreenewi}. For any $\theta$
\begin{eqnarray*}
\theta a-h(\theta) &=& \theta\gamma-h(\theta)
-\theta(\gamma-a) \\
&\leq&
\phi\gamma-h(\phi) -\theta(\gamma-a)
\\
&=& \phi a-h(\phi) -
(\theta-\phi) (\gamma-a),
\end{eqnarray*}
and so, when $(\theta-\phi) (\gamma-a)\geq0$,
$
\theta a-h(\theta)
\leq\phi a-h(\phi) $.
Hence,
for $a \leq\gamma$
\[
h^{\ast}(a)= \sup_\theta\{\theta a - h(\theta) \}=
\sup_{\theta\leq\phi} \{\theta a - h(\theta) \}
\]
and this holds also for $g$, giving~\ref{rateagreenewii}. Also, for
$a > \gamma$,
\[
\sup_{\theta\leq\phi} \{\theta a - h(\theta)\}
=\phi a-h(\phi)
=\phi\gamma-h(\phi)-\phi(\gamma-a)=h^{\ast}(\gamma)-\phi(\gamma-a)
\]
and when $g(\theta)=\infty$ for $\theta>\phi$ the first expression here
is $g^{\ast}(a)$.
\end{pf}
%
%
\begin{lem}\label{f-natural} \label{f-flt}
Let $f$ be $k$-convex, $\Gamma= \Gamma(f^{\ast})$, and
$\varthetapps
{f}=\vartheta^{\flat}(f)$.
{\renewcommand\thelonglist{(\roman{longlist})}
\renewcommand\labellonglist{\thelonglist}
\begin{longlist}
\item\label{f-flti}If $\Gamma>-\infty$ and $\subdiff f^{\ast}(
\Gamma
)= \varnothing$ or $f^{\ast}(\Gamma)<0$,
then $f^{\flat}=\cl{f}$ and $\varthetapps{f}=\infty$.
\item\label{f-fltii}
If $\Gamma>-\infty$ and $\subdiff f^{\ast}(\Gamma)\neq\varnothing
$, then
for any $\phi\in\subdiff f^{\ast}(\Gamma)$
\[
f^{\flat}(\theta)= \cases{
\cl{f}(\theta),&\quad$\theta\leq\phi$,\cr
\theta\Gamma- f^{\ast}(\Gamma), &\quad$\theta\geq\phi$.}
\]
\item\label{f-fltiii}
$f^{\flat}(\theta)=\cl{f}(\theta)$ if and only if
$\theta\leq\varthetapps{f}$.
\end{longlist}}
\end{lem}
\begin{pf}
Assume $\subdiff f^{\ast}(\Gamma)= \varnothing$. Then
$f^{\ast}(a)=\infty$ for $a > \Gamma$,
using \citet{MR0274683}, Theorem 23.4.
Also, if
$f^{\ast}(\Gamma)<0$, then, since $f^{\ast}$ is continuous when finite,
$f^{\ast}(a)=\infty$ for $a>\Gamma$. Hence, in both cases,
\[
f^{\flat}(\theta)=\sup_{a \leq\Gamma} \{ \theta a - f^{\ast}(a)\}
=\sup_{a}
\{ \theta a - f^{\ast}(a)\}
=\cl{f}(\theta),
\]
and
so $\vartheta^{\flat}(f)=\inf\{\theta\dvtx f^{\flat}(\theta)<\cl
{f}(\theta
)\}
=\infty$.
This give~\ref{f-flti}.
Now assume $\subdiff f^{\ast}(\Gamma)\neq\varnothing$.
For any $\phi\in\subdiff f^{\ast}( \Gamma)$, Lemma~\ref{rateagreenew}
(with $h=f^{\ast}$ and $g=f^{\acci}$) gives~\ref{f-fltii} because
$(f^{\ast})^{\ast}=\cl{f}$.

Turning to the final part, the result is immediate (and without real
content) when $\Gamma=-\infty$. It also holds when~\ref{f-flti} holds.
When~\ref{f-fltii} holds $\vartheta^{\flat}(f) \geq\sup\subdiff
f^{\ast}(\Gamma)$, but when $\cl{f}(\phi)=f^{\flat}(\phi)
=\phi\Gamma-f^{\ast}(\Gamma)$ Lemma
\ref{subdiffeasy}\ref{subdiffeasyii} gives $\phi\in \subdiff
f^{\ast}(\Gamma)$. Hence $\vartheta^{\flat}(f) = \sup\subdiff
f^{\ast}(\Gamma)$ and $f^{\flat}(\theta)< \cl{f}(\theta)$ for all
$\theta
>\vartheta^{\flat}(f)$.
\end{pf}
\begin{pf*}{Proof of Lemma~\ref{nat=flt}}
Let $\vartheta=\vartheta^{\flat}(f)$ and $\Gamma=\Gamma(f^{\ast})$.
When $\Gamma=-\infty$,
$f^{\ast}(a)>0 $ for all $a$, $f^{\flat}\equiv-\infty$ and
$\vartheta
=-\infty$. If $\nat{f}\not\equiv-\infty$, then, for some finite $A
\geq0$ and $B$,
$A+B\theta\leq\nat{f}(\theta) \leq f(\theta)$
and then $f^{\ast}(B)\leq-A \leq0$.
Hence when $\Gamma=-\infty$, $\nat{f}\equiv-\infty$ and
$\varthetapp
{f}=-\infty$.

Assume now that $\Gamma>-\infty$, so that $\nat{f}\not\equiv
-\infty$.
Then $\nat{f}(\psiu)=f(\psiu)$.
By Lem\-ma~\ref{f-natural1}\ref{f-natural1-ii}, $\nat{f} \geq f^{\flat}$
and using Lemma~\ref{f-natural}
$f^{\flat}(\psiu)=\cl{f}(\psiu)\leq f(\psiu)=\nat{f}(\psiu)$.
We need to show that
$\nat{f}$ and $f^{\flat}$ agree on $(\psiu, \infty)$.
When $\PhiD( f)=\{\psiu\}$
the result holds. Hence we may suppose $\PhiD( f)$ has a
nonempty interior.
Then
$f \geq\nat{f}
\geq f^{\flat} = \cl{f}=f$ on $(\psiu,\vartheta)$.
Thus the result holds when $\vartheta=\infty$, and so we can assume
$\vartheta<\infty$, and hence, by Lemma~\ref{f-natural}\ref{f-flti},
that $f^{\ast}(\Gamma)=0$.
Then, by Lemma~\ref{f-natural}\ref{f-natural1-ii},
$f^{\flat}(\theta)=
f(\theta)$ for
$\theta\in(\psiu, \vartheta)$ and $f^{\flat}(\theta)=\Gamma
\theta$
for $\theta\in[\vartheta, \infty)$.
Suppose that for some $\phi>\psiu$, $\nat{f}(\phi)>f^{\flat}(\phi)$.
Hence, $\phi\geq\vartheta$ and $\nat{f}(\phi)> \Gamma\phi$. Then
\[
\frac{\nat{f}(\phi)}{\phi}>\Gamma=\frac{f^{\flat}(\vartheta
)}{\vartheta}=
\frac{\cl{f}(\vartheta)}{\vartheta}=
\liminf_{\theta\rightarrow\vartheta}
\frac{{f}(\theta)}{\theta}\geq
\liminf_{\theta\rightarrow\vartheta}
\frac{\nat{f}(\theta)}{\theta},
\]
contradicting that
$\nat{f}(\theta)/\theta$ is decreasing and continuous at $\phi$.

It remains to prove $\varthetapp{f}=\vartheta$ in this case.
Lemma~\ref{f-natural}\ref{f-fltiii} gives
\[
\vartheta=\inf\{\theta\dvtx f^{\flat}(\theta)<\cl{f}(\theta)\}
=\sup\{\theta\dvtx f^{\flat}(\theta)=\cl{f}(\theta)\}
\]
and the relationship between $\nat{f}$ and $f^{\flat}$ already
established means
this equals $\sup\{\theta\dvtx\nat{f}(\theta)=f(\theta)\}$ which is
$\varthetapp{f}$.\vspace*{-2pt}
\end{pf*}
\begin{pf*}{Proof of Proposition~\ref{f-natural2}}
This uses Lemmas~\ref{nat=flt} and~\ref{f-natural}.
When $\Gamma=-\infty$, Lemma~\ref{nat=flt} contains the result.
When $\subdiff f^{\ast}(\Gamma)=\varnothing$
or $f^{\ast}(\Gamma)<0$ the characterization of $\nat{f}$
follows from Lemma~\ref{f-natural}\ref{f-flti}. In the
remaining cases $\vartheta=\varthetapp{f}<\infty$ and the
characterization follows from
Lemma~\ref{f-natural}\ref{f-fltii}. The assertion about $\Gamma$
follows from this characterization.\vspace*{-2pt}
\end{pf*}

The following lemma will be important in later sections. The
one after it records various facts needed to prove
the alternative recursion in Proposition~\ref{alternativerec}.\vspace*{-2pt}
%
%
\begin{lem}\label{f-fltiiii}
Let $f$ be $k$-convex and $a \in\subdiff\cl{f}(\theta)$.
{\renewcommand\thelonglist{(\roman{longlist})}
\renewcommand\labellonglist{\thelonglist}
\begin{longlist}
\item\label{f-fltiiiii} If $\theta> \varthetapp{f}$, then $ f^{\ast}(a)>0$.
\item\label{f-fltiiiiii} If $\theta< \varthetapp{f}$, then
$ f^{\ast}(a) \leq0$.\vspace*{-2pt}
\end{longlist}}
\end{lem}
\begin{pf}
By Lemma~\ref{nat=flt}, $\varthetapp{f}=\vartheta^{\flat}(f)$.
Lemma~\ref{subdiffeasy} gives
\[
\cl{f}(\theta)=\theta a - f^{\ast}(a)
=\sup_b \{\theta b - f^{\ast}(b)\}\geq\sup_{b \leq
\Gamma}
\{\theta b - f^{\ast}(b)\}=f^{\flat}(\theta).
\]
When $\theta> \varthetapp{f}$ there is strict inequality, implying
that $f^{\ast}(a)>0$.

If $0=\theta< \varthetapp{f}$, then $\Gamma(f^{\ast})>-\infty$
and so
$f^{\ast}(a)=-\cl{f}(0) <0$. Otherwise, take $\theta<\theta
+\varepsilon
<\varthetapp{f}$.
Note that $f^{\flat}(\theta)/\theta$ is decreasing on $(0,\infty)$ and
equals $\cl{f}(\theta)/\theta$ on $(0,\varthetapp{f})$,
and that $f^{\flat}(\theta)=\cl{f}(\theta)=\theta a - f^{\ast
}(a)$. Therefore
\[
\frac{\theta+ \varepsilon}{\theta}\bigl( \theta a - f^{\ast
}(a)
\bigr)=\frac
{\theta+ \varepsilon}{\theta}\cl{f}(\theta)\geq\cl{f}(\theta
+\varepsilon)
\geq(\theta+\varepsilon) a - f^{\ast}(a).
\]
Thus $-\varepsilon f^{\ast}(a)/\theta\geq0$.
\end{pf}
%
%
\begin{lem}\label{fnatsimple}
Suppose $f$ and $\kappa$ are
$k$-convex.
{\renewcommand\thelonglist{(\roman{longlist})}
\renewcommand\labellonglist{\thelonglist}
\begin{longlist}
\item\label{fnatsimplei}
$f^{\acci}=(\nat{f})^{\acci}=(\nat{f})^{\ast}$ and
$\cl{\nat{f}}=(f^{\acci})^{\ast}$;
\item\label{fnatsimpleii} $\PhiD( \nat{f})=\Phidp{f}$;
\item\label{fnatsimpleiii}\label{sweeplem} $\nat{{\mathfrak
M}[\nat{f},\nat{\kappa}]}=
{\mathfrak M}[\nat{f},\nat{\kappa}] \leq{\mathfrak M}[\nat
{f},\kappa] $.
\end{longlist}}
\end{lem}
\begin{pf}
The first part follows easily from Lemmas
\ref{firstconvlem} and~\ref{nat=flt}, because $f^{\flat}=(f^{\acci
})^{\ast}$,
and the second from Lemmas~\ref{nat=flt} and~\ref{f-natural1}\ref
{f-natural1-iii}.
For the final one, just note that ${\mathfrak M}[\nat{f},\nat{\kappa}]$
inherits all the right properties from $\nat{f}$ and $\nat{\kappa}$.\vspace*{-2pt}
\end{pf}
\begin{pf*}{Proof of Proposition~\ref{alternativerec}}
By definition~(\ref{firstrecur}),
$f^{\acci}_1=\kappa^{\acci}_1=r_1$.
Suppose the result is true for
$i-1$. By Lemmas~\ref{kh}(ii) and~\ref{fnatsimple}\ref{fnatsimplei}
\begin{eqnarray*}
(\nat{f}_i)^{\ast}&=&f^{\acci}_i=
{\mathfrak M}[\nat{f}_{i-1},\kappa_i]^{\acci}
=({\mathfrak M}[\nat{f}_{i-1},\kappa_i]^{\ast})^{\circ}
\\[-2pt]
&=&{\mathfrak C}[f^{\acci}_{i-1},\kappa^{\ast}_i]^{\circ}
={\mathfrak C}[r_{i-1},\kappa^{\ast}_i]^{\circ}=r_i
\end{eqnarray*}
as required.\vadjust{\goodbreak}
\end{pf*}
%
%
\begin{lem}\label{alternativerectwo}
Let $f_i$ be given by~(\ref{secondrecur}).
When~(\ref{phiiorder}) holds, $f_i$ is closed and $k$-convex,
$ [\phi_i, \infty) \subset\PhiD( \nat{f}_{i})= \bigcap
_{j \leq i} \Phidp
{\kappa_j},
$ $-\infty<r_i$ for each $i$, and if $f_1(0)>0$, then $f_i(0)>0$.
\end{lem}
\begin{pf}
Using Lemma~\ref{analytic}, $f_1=\kappa_1$ is $k$-convex,
and by Lemma~\ref{fnatsimple}\ref{fnatsimpleii}
$\PhiD( \nat{f}_1)= \Phidp{\kappa_1}$.
Hence the result is true for $i=1$. Suppose the result holds for $i-1$.
By definition,
\[
\PhiD( f_i)=
\PhiD( {\mathfrak M}[\nat{f}_{i-1},\kappa_i])=\PhiD
(
\nat{f}_{i-1})
\cap\PhiD( \kappa_i) \supset
[\phi_{i-1},\infty)\cap\PhiD( \kappa_i),
\]
which is nonempty, since it contains $\phi_i$ by~(\ref{phiiorder}).
Thus $f_i$ is $k$-convex and $\PhiD( \nat{f}_i)$ contains $[\phi
_i,\infty)$. Furthermore, $\nat{f}_{i-1}$ and $\kappa_i$ are closed, so
$f_i$ is, too. Since $\PhiD( f_i)$ is nonempty $ \Phidp{{f}_i} = \PhiD(
\nat{f}_{i-1}) \cap\Phidp{\kappa_i}, $ and then the induction
hypothesis and Lemma~\ref{fnatsimple}\ref {fnatsimpleii} confirm the
formula for $\PhiD( \nat{f}_{i})$. Now, by Lemma~\ref{covex-d}(i),
$-\infty<(\nat{f}_i)^{\ast} =f^{\acci}_i=r_i$. Since $f_{i-1}$ is
closed, ${f}_{i-1}(0)>0$ implies that $\nat {f}_{i-1}(0)={f}_{i-1}(0)$
and then $f_i(0)\geq\nat {f}_{i-1}(0)={f}_{i-1}(0)>0$.
\end{pf}

\section{Upper bounds on numbers}\label{upper}
Here, Theorem~\ref{newub} will be proved.
The first lemma presses the argument deployed at
the start of the proof of Proposition~\ref{oneclass} a little further.
It notes that~(\ref{induct12})
implies the apparently stronger~(\ref{induct13}).
The minor distinction between $\nat{f}$ and $f^{\flat}$ $(\mbox{$=$}(f^{\acci
})^{\ast})$,
exposed in Lemma~\ref{nat=flt}, matters in this result.
%
%
\begin{lem}\label{upperlemma}
Suppose that for a $k$-convex $f$ with $\Gamma(f^{\ast})> -\infty$
and a
point processes $P^{(n)}$
%
%
\begin{equation}\label{induct12}
\limsup_n \frac{1}{n} \log
\biggl(\int e^{\theta x}P^{(n)}(dx)
\biggr) \leq f(\theta) \qquad\mbox{a.s. }\forall\theta.
\end{equation}
Then
%
%
\begin{equation}\label{induct11}
\limsup_n \frac{1}{n} \log\bigl( P^{(n)}[na,\infty) \bigr)
\leq-f^{\acci}(a)\qquad\mbox{a.s. }\forall a
\end{equation}
and
%
%
\begin{equation}\label{induct13}
\limsup_n \frac{1}{n} \log
\biggl(\int e^{ \theta x}P^{(n)}(dx)
\biggr) \leq\nat{f}(\theta) \qquad\mbox{a.s. }\forall\theta.
\end{equation}
\end{lem}
\begin{pf}
For $\theta\geq0$,
\[
\theta na + \log P^{(n)}[na,\infty) \leq\log\int e^{ \theta x}P^{(n)}(dx)
\]
and so using~(\ref{induct12}),
minimizing over $\theta$, and using that $P^{(n)}[na,\infty)$
is eventually zero when it decays gives~(\ref{induct11}).
The assertions~(\ref{induct12}) and~(\ref{induct13})
are the same when\vadjust{\goodbreak} $\varthetapp{f}=\infty$.
Hence we may assume $\varthetapp{f}<\infty$. For $\varepsilon>0$ and
large enough $n$,
$P^{(n)}[n(\Gamma(f^{\ast})+\varepsilon),\infty)=0$. Then, for $
\theta
\geq
\psi$,
\[
\int e^{ \theta x}P^{(n)}(dx)
\leq e^{(\theta-\psi) (\Gamma(f^{\ast})+\varepsilon)n}\int e^{ \psi
x}P^{(n)}(dx)
\]
so that~(\ref{induct12}) gives
\[
\limsup\frac{1}{n} \log
\biggl(\int e^{ \theta x}P^{(n)}(dx)
\biggr) \leq f(\psi)+(\theta-\psi) \Gamma(f^{\ast}) \qquad\mbox{a.s.}
\]
Take $\psi=\theta$
when $\theta< \varthetapp{f}$ and when $\theta= \varthetapp{f}=\inf
\PhiD( f)$,
so in these cases the right-hand side is just $f(\theta)$.
Otherwise, take $\psi\in\PhiD( f)$ and
then let $\psi\rightarrow\varthetapp{f}$.
[If $f$ is lower semi-continuous at $\varthetapp{f}$, taking
$\psi=\varthetapp{f}$ will do.] Then the right-hand side becomes
$\cl{f}(\varthetapp{f})+(\theta-\varthetapp{f}) \Gamma(f^{\ast}) $.
Proposition~\ref{f-natural2}
confirms that the right-hand side is $\nat{f}$ in all cases.
\end{pf}

Recall that
$-\chi_i$ is the logarithm of the indicator function of the set $\PhiD
_{i-1,i}$.

%
\begin{lem}\label{Fgood-tran2}
In a sequential process with $m_{\upsilon\tau}> 0$ for $\upsilon\in
\pen$ and $\tau\in\last$,
suppose that for all $\nu\in\conC{1}$
and $\theta$
\[
\limsup\frac{1}{n} \log
\biggl(\int e^{ \theta x}Z^{(n)}_{\upsilon}(dx)
\biggr) \leq f (\theta)\qquad\mbox{a.s.-}\tP_{ \nu},
\]
where $f$ is $k$-convex with $\Gamma(f^{\ast})> -\infty$.
Let
$g=\nat{f}+ \chi_K$ and let $\kappa$ be the \PF eigenvalue of the final
block in
$m$, corresponding to $\last$.
Then, for \mbox{$\sigma\in\last$},
\[
\limsup\frac{1}{n} \log
\biggl(\int e^{ \theta x}Z^{(n)}_{\sigma}(dx)
\biggr) \leq\nat{ {\mathfrak M}[\nat{g},\kappa]} (\theta)
\qquad\mbox{a.s.-}\tP_{
\nu}
\]
and $\Gamma(\nat{ {\mathfrak M}[\nat{g},\kappa]})> -\infty$.
\end{lem}
\begin{pf} Note first that $\nat{f} \leq\nat{g} \leq\nat{
{\mathfrak M}[\nat{g},\kappa]}$, so that $\Gamma(f^{\ast})>
-\infty$ implies that
$\Gamma(g^{\ast})> -\infty$ and that $\Gamma(\nat{ {\mathfrak
M}[\nat{g},\kappa]})> -\infty$.

Recall that $F^{(n)}_{\tau}$ are those in the $n$th generation that are
the first of type $\tau$ in their line of descent.
Taking conditional expectations,
\[
\tE\biggl[
\int e^{ \theta x}F^{(n+1)}_{\tau}(dx)
\Big| \cF{n}\biggr]
=
\biggl(\int e^{\theta x}Z^{(n)}_{\upsilon}(dx)\biggr)
m_{\upsilon\tau}(\theta)
\]
and so, using Lemma~\ref{upperlemma} and the definition of $g$,
\[
\limsup\frac{1}{n} \log\tE\biggl[
\int e^{ \theta x}F^{(n+1)}_{\tau}(dx)
\Big| \cF{n}\biggr] \leq g(\theta) \qquad\mbox{a.s.-}\tP_{ \nu}.
\]
Then conditional Borel--Cantelli [e.g., \citet{MR496420}] gives that
\[
\limsup\frac{1}{n} \log
\biggl(\int e^{\theta x}F^{(n)}_{\tau}(dx)
\biggr) \leq g(\theta) \qquad\mbox{a.s.-}\tP_{ \nu}
\]
and a further application of Lemma~\ref{upperlemma} gives that
\[
\limsup\frac{1}{n} \log
\biggl(\int e^{ \theta x}F^{(n)}_{\tau}(dx)
\biggr) \leq\nat{g}(\theta) \qquad\mbox{a.s.-}\tP_{ \nu}.
\]

The set of particles obtained as those
first in their lines of descent that are either in $\last$ or in
generation $n$ forms an optional line, as in \citet{MR1014449}.
Let $\cG{n}$ contain all information on reproduction down lines of
descent to particles in this line.
In this sequential process the first in any line of descent
with a type in $\last$ is necessarily of type $\tau$.
For any $\sigma\in\last$ and $\theta$,
\[
\tE\biggl[
\int e^{ \theta x}Z^{(n)}_{\sigma}(dx)
\Big| \cG{n}\biggr]
=
\sum_{r=0}^{n}
\int e^{ \theta x}F^{(r)}_{\tau}(dx)
( m(\theta)^{n-r})_{\tau\sigma}.
\]
Hence, the bound just obtained, Lemma~\ref{matrix}, and routine
estimation give
\[
\limsup_n
\frac{1}{n}
\log
\tE\biggl[
\int e^{\theta x}Z^{(n)}_{\sigma}(dx)
\Big| \cG{n}\biggr]
\leq
{\mathfrak M}[\nat{g},\kappa](\theta)
\qquad\mbox{a.s.-}\tP_{ \nu}.
\]
Conditional Borel--Cantelli and Lemma~\ref{upperlemma} complete the
proof.
\end{pf}
%
%
\begin{lem}\label{niceg} Define $g_i$ by~(\ref{thirdrecur}). Then
$g_K$ is finite somewhere on $(0,\infty)$
if and only if~(\ref{phiiorder}) holds and~(\ref{newphiiorder})
holds for $i=1,2,\ldots,K - 1$. When these hold $g_K$ is $k$-convex,
\[
[\phi_K, \infty) \subset\PhiD( \nat{g}_K)=
\biggl(\bigcap_{j \leq K}
\Phidp{\kappa_j}\biggr)
\cap\biggl(\bigcap_{j \leq K - 1} \PhiD^+_{j,j+1} \biggr),
\]
$\nat{g}_K$ is continuous on $\PhiD( \nat{g}_K)$, and
$-g^{\ast}_K(a)< \infty$ for some finite $a$.
\end{lem}
\begin{pf}
Assume $g_K(\phi_K)$ is finite. Then
$\phi_K \in\PhiD( \kappa_K)$ and there is a
$\phi_{K - 1,K} \leq\phi_K$ such that $ (\nat{g}_{K - 1}+\chi
_{K})(\phi_{K - 1,K})<\infty$,
which implies that $\phi_{K - 1,K} \in\PhiD_{K - 1,K}$ and that
there is a
$\phi_{K - 1}\leq\phi_{K - 1,K}$ with $g_{K - 1}(\phi_{K -
1})$ finite. Hence, by induction on $K$,
$g_K(\phi)$ finite for some positive $\phi$ implies that (\ref
{phiiorder}) holds and~(\ref{newphiiorder}) holds for $i=1,2,\ldots
,K
- 1$.

Now suppose~(\ref{phiiorder}) holds and~(\ref{newphiiorder}) holds
for $i=1,2,\ldots,K - 1$.
All the assertions of the lemma then hold with $g_1=\kappa_1$ in place
of $g_K$.
Suppose all the assertions hold for $g_{K - 1}$.
Then
\[
\PhiD( \nat{g}_{K - 1}+\chi_{K - 1})
=\Phidp{g_{K - 1}}\cap\PhiD_{K - 1,K}
\supset[\phi_{K - 1}, \infty)\cap\PhiD_{K - 1,K}\ni\phi_{K
- 1,K}.
\]
Since this is nonempty,
\[
\PhiD( g_K)=\PhiD\bigl( {\mathfrak M}[\nat{(\nat
{g}_{K - 1}+\chi_{K - 1})},\kappa_K]\bigr)
=
\Phidp{g_{K - 1}}\cap\PhiD^+_{K - 1,K} \cap\PhiD(
\kappa_K)
\]
and $g_K$ is continuous there, because $\nat{g}_{K - 1}$ is by
assumption and
$\kappa_K$ is by Lem\-ma~\ref{analytic}.
Furthermore
$
\PhiD( g_K)\supset
[\phi_{K - 1}, \infty) \cap\PhiD( \kappa_K) \ni\phi_K
$ and so is nonempty.
Then, using Lemma~\ref{fnatsimple}\ref{fnatsimpleii},
\[
\PhiD( \nat{g}_K)=\Phidp{g_K} = \Phidp{g_{K - 1}}\cap
\PhiD^+_{K -
1,K} \cap\Phidp{\kappa_K}\supset[\phi_K,\infty),\vadjust{\goodbreak}
\]
and $\nat{g}$ is continuous there. Substituting for $\Phidp{g_{K - 1}}$
gives the formula for $\Phidp{g_K}$. Lemma~\ref{covex-d}\ref
{covex-di} gives the final part and the induction is complete.
\end{pf}
\begin{pf*}{Proof of Theorem~\ref{newub}}
Note first that the final assertion is contained in Lemma~\ref{niceg}.
Now, by Lemma~\ref{upperlemma}, it is enough to show that
\[
\limsup\frac{1}{n} \log
\biggl(\int e^{ \theta x}Z^{(n)}_{\sigma}(dx)
\biggr) \leq g_K(\theta)\qquad\mbox{a.s.-}\tP_{ \nu}
\]
and that $\Gamma(g_K^{\ast})>-\infty$.
Both hold when $K=1$, the first by Lemma~\ref{transbound}, the second
by combining Lemmas~\ref{covex-d}\ref{covex-dvi}, \ref
{analytic}\ref
{analyticiv} and
the assumption that $\kappa_1(0)>0$.
Assume the result holds for $K - 1$. Then it holds also for $K$, by
Lemma~\ref{Fgood-tran2}
with $f=g_{K - 1}$ and $\kappa=\kappa_K$.
\end{pf*}

\section{Matching the lower and upper bounds}\label{firstproof}

In this section Theorems~\ref{maintheorem} and~\ref{f=g}
will be proved, using Theorem~\ref{newub}. These are cases where the
upper bound on numbers matches the lower bound based on Theorem \ref
{prelimmaintheorem}. The simpler theorem will be discussed first.
\begin{pf*}{Proof of Theorem~\ref{maintheorem}}
Let $f_i$ and $g_i$ be as~(\ref{secondrecur}) and~(\ref{thirdrecur}).
Clearly $g_1=f_1=\kappa_1$. Assume $g_{i-1}=f_{i-1}$.
Note first that $\nat{(\nat{f}_{i-1} + \chi_i)}\geq\nat{f}_{i-1}$
and so
\[
g_i={\mathfrak M}[\nat{(\nat{g}_{i-1} + \chi_i)},\kappa_i] =
{\mathfrak M}[\nat{(\nat{f}_{i-1} + \chi_i)},\kappa_i] \geq
{\mathfrak M}[\nat{f}_{i-1},\kappa_i] = f_i.
\]
By Lemma~\ref{alternativerectwo},~(\ref{off-diag-cond})
is equivalent to
$\PhiD( \nat{f}_{i-1})\cap\PhiD( \kappa_{i})
\subset\PhiD_{i-1,
i}$ $(\mbox{$=$}\PhiD( \chi_i))$,
and when this holds
$
{\mathfrak M}[\nat{f}_{i-1} + \chi_i,\kappa_i]={\mathfrak M}[\nat
{f}_{i-1},\kappa_i]$.
Then,
\[
g_i ={\mathfrak M}[\nat{(\nat{f}_{i-1} + \chi_i)},\kappa_i]\leq
{\mathfrak M}[\nat{f}_{i-1} + \chi_i,\kappa_i] ={\mathfrak
M}[\nat{f}_{i-1},\kappa_i]=f_i.
\]
Hence
$g_i=f_i$. Thus, by induction, $g_K=f_K$. Then $g^{\acci}_K=f^{\acci}_K$,
which by Corollary~\ref{newmain} gives the result.
\end{pf*}

The proof just given relies on a simple estimation of
$\nat{(\nat{f}_{i-1} + \chi_i)}$ and then $\PhiD(
\kappa_{i})$ making
$\chi_i$ irrelevant.
To deal with more cases it is necessary to refine the estimation of
$\nat{(\nat{f}_{i-1} + \chi_i)}$
and make a more careful comparison of the result with~$\kappa_i$. This
is done next.
\label{evenmore}
%
%
\begin{lem}\label{fgagreelemma}Suppose $f$ and $\kappa$ are $k$-convex
with $\Gamma(f^{\ast})>-\infty$. Suppose~$C$ is a convex set, and let
$\chi(\theta)=-\log I(\theta\in C)$, $\psiu=\inf C$ and $\psio
=\sup
C$. Let $\chi_1(\theta)=-\log I(\theta\in C^+)$ and $\chi_2(\theta)
=-\log I(\theta\in(-\infty, \psio])$.
{\renewcommand\thelonglist{(\roman{longlist})}
\renewcommand\labellonglist{\thelonglist}
\begin{longlist}
\item\label{fgagreelemmanewi} $\Gamma({\mathfrak M}[\nat{(\nat{f}
+ \chi)},\kappa]^{\ast})>-\infty$.
\item\label{fgagreelemmai}
If $\PhiD( \nat{f})\cap C \neq\varnothing$ and $\nat{f}$ is
continuous from the right at $\psio$,
then
\[
\nat{(\nat{f} + \chi)}(\theta)=
\cases{
(\nat{f}+\chi)(\theta),&\quad$\theta< \psio$, \vspace*{2pt}\cr
\theta\bigl(\nat{f}(\psio)/\psio\bigr),&\quad$\theta\geq\psio$.}\vadjust{\goodbreak}
\]
\item\label{fgagreelemmaii}
If, in addition to the conditions in~\ref{fgagreelemmai},
%
%
\begin{equation}\label{eitheror}
\mbox{either }\kappa(\theta) \geq\theta\bigl(\nat{f}(\psio)/\psio\bigr)
\mbox{ for }
\theta
\in[\psio, \infty)
\quad\mbox{or}\quad
\varthetapp{f}\leq\psio,
\end{equation}
then
\[
{\mathfrak M}[\nat{(\nat{f} + \chi)},\kappa]={\mathfrak M}[\nat
{f}+\chi_1,\kappa].
\]
\item\label{fgagreelemmaiii}
If, in addition to the conditions in~\ref{fgagreelemmai},
$\PhiD( \nat{f})\cap\PhiD( \kappa) \subset
[\psiu, \infty)$,
then
\[
{\mathfrak M}[\nat{(\nat{f} + \chi)},\kappa]={\mathfrak M}[\nat
{(\nat{f}+\chi_2)},\kappa],
\]
except possibly at $\psiu$, and when they differ there the left-hand
side is infinite.
\item\label{fgagreelemmaiv}
When the conditions in both
\ref{fgagreelemmaii}
and~\ref{fgagreelemmaiii} hold, ${\mathfrak M}[\nat{(\nat{f} +
\chi)},\kappa]={\mathfrak M}[\nat{f},\kappa]$
except possibly at $\psiu$, and when they differ there the left-hand
side is infinite.
\end{longlist}}
\end{lem}
\begin{pf} The proof of part~\ref{fgagreelemmanewi} mimics the
first part of the proof of Lem\-ma~\ref{Fgood-tran2}.
The form of $\nat{(\nat{f} + \chi)}$ in~\ref{fgagreelemmai}
follows from Proposition~\ref{f-natural2}.
Now, assume~(\ref{eitheror}) holds. In the first case, $\nat{(\nat
{f}
+ \chi)}$
is dominated by
$\kappa$ in $[\psio,\infty)$ and equals $\nat{f}$ on~$C$. In the second,
since $\varthetapp{f}\leq\psio<\infty$ and $\nat{f}$ is continuous
from the right at $\psio$,
$\Gamma(f^{\ast})=\nat{f}(\psio)/\psio$ by Proposition \ref
{f-natural2}; and so
$\nat{(\nat{f} + \chi)}=\nat{f}$ on $C^+$,
and this also holds when $\psio=\infty$.
Hence in both cases
$
{\mathfrak M}[\nat{(\nat{f} + \chi)},\kappa]={\mathfrak M}[\nat
{f}+\chi_1,\kappa]
$, proving~\ref{fgagreelemmaii}.
By~\ref{fgagreelemmai},
$\nat{(\nat{f} + \chi)}$ and $\nat{(\nat{f} + \chi_2)}$
agree for $\theta\geq\psio$, and $\nat{(\nat{f} + \chi_2)}=\nat{f}$
for $\theta< \psio$.
Since
$\PhiD( {\mathfrak M}[\nat{f},\kappa])=\PhiD(
\nat
{f})\cap\PhiD( \kappa),
$
${\mathfrak M}[\nat{(\nat{f} + \chi)},\kappa]$ and ${\mathfrak
M}[\nat{f},\kappa]$
agree
(and are both infinite) on $(-\infty, \psiu)$ and by~\ref{fgagreelemmai}
they agree on $(\psiu,\psio)$. They also agree at
$\psiu$ when $\psiu\in C$ and when it is not $(\nat{f} + \chi)$
is infinite there. This proves~\ref{fgagreelemmaiii}.
The final part is an application of~\ref{fgagreelemmaiii} to
$f+\chi_1$.
\end{pf}
\begin{pf*}{Proof of Theorem~\ref{f=g}}
Note first that,
by Lemma~\ref{fnatsimple}\ref{fnatsimpleii},
$\Phidp{g_{K - 1}}=\PhiD( \nat
{g}_{K - 1})$.
Also, Lemmas~\ref{fnatsimple}\ref{fnatsimpleii} and \ref
{alternativerectwo} show that the left of~(\ref{ii}) is just
$\Phidp{f_i} \cap\PhiD( \kappa_{i+1}) $.

The proof is by induction. For it, add in the additional assertion that
$\nat{g}_{K}=\nat{f}_{K}$, except possibly at $\inf\PhiD( f_K)$ when
$\nat{g}_{K}$ is infinite there. The result, including this additional
assertion, is true for $K=1$. Assume the result and the addition are
true for $K - 1$. When~(\ref{phiiorder}) holds and~(\ref{newphiiorder})
holds for $i=1,2,\ldots,K - 1$, Lemma~\ref{niceg} implies that
$\nat{g}_{K - 1}$ is finite at $\psio_{K - 1}$ and so equals
$\nat{f}_{K - 1}$ and is continuous from the right there. Also, by the
induction hypothesis $\PhiD( \nat{g}_{K - 1}) \subset\PhiD( \nat {f}_{K
- 1})$ [and equals it unless $\nat{f}_{K - 1}$ is finite and
$\nat{g}_{K - 1}$ infinite at $\inf\PhiD( \nat{f}_{K - 1})=\inf\PhiD (
f_{K - 1})$]. Hence~(\ref{eitheror1}) and~(\ref{ii}) with $i=K - 1$
mean Lemma~\ref{fgagreelemma}\ref{fgagreelemmaiv} applies. Together
with the induction hypothesis this gives\vspace*{-1pt} $g_K={\mathfrak
M}[\nat{f}_{K - 1},\kappa_K]=f_K$ except possibly at $\psiu_{K - 1}$
and $\inf\PhiD( \nat {f}_{K - 1})$, where they can only differ with
$g_K$ being infinite. Furthermore, by Lemma~\ref{niceg},
$f_K(\phi_K)\leq g_K(\phi_K)<\infty$. Since both functions are proper
and convex, and $f_K$ is closed, they can only differ by $g_k$ being
greater, and infinite, at the endpoints of $\PhiD( f_K)$. Hence
$\nat{g}_K=\nat{f}_K$ except possibly at $\inf\PhiD( \nat{f}_K)$. Then
these two functions have the same F-dual, that is,
$g^{\acci}_K=f^{\acci}_K$.
\end{pf*}

\section{Formulas for the speed}\label{formulaforGamma}
The main objective here is to
establish Theorem~\ref{gformullem1} giving an alternative formula for
the speed $\Gamma(g^{\ast}_K)$, which plays a critical role
in the proof of Theorem~\ref{bestspeed}. A few other remarks are also
included about
computing the speed.

There are several alternative formulas for $\Gamma(f^{\ast})$
from the irreducible case that apply more widely
to any $k$-convex $f$. One is contained
in~(\ref{w-speed}) in Proposition~\ref{f-natural2}.
Another is that
$
\Gamma
=\sup\{a\dvtx f^{\ast}(a) \leq0\}
$,
which holds because $f^{\ast}$ is convex and increasing.
Furthermore, by convexity $\Gamma$ is the unique solution to $
f^{\ast}(\Gamma)=0,
$ provided only that there are a $u$ and $v$ with
$f^{\ast}(u)< 0 \leq f^{\ast}(v)< \infty$.

When $f$ is differentiable throughout $\PhiD( f)$ and
there is a $\theta$ such that $
\theta f'(\theta)-f(\theta)=0
$,
then $\Gamma(f^{\ast})=f'(\theta)$---this is straightforward
calculus when $\theta$ is in the interior of $\PhiD( \kappa
)$, and all
cases are covered by
\citet{MR0274683}, Theorem 23.5(b). Then $\Gamma(f^{\ast})$
can be
found by solving
$f(\theta)=\theta f'(\theta)
$ for $\theta$. This is certainly relevant in the irreducible
case,
since Lemma~\ref{analytic}\ref{analyticiii} gives that $f=\kappa$ is
differentiable, but need not be once there is more than one class.
%
%
\begin{lem}\label{gtot}
Suppose that $f$ and $\kappa$ are $k$-convex with $\Gamma(f^{\ast
})>-\infty
$, that $\chi=-\log I(\theta\in C)$ for a convex $C$,
that $g={\mathfrak M}[\nat{(\nat{f}+\chi)},\kappa]$ and that
this $g$ is
finite somewhere
[so $\PhiD( \nat{f}) \cap C \cap\PhiD( \kappa)
\neq\varnothing$].
Let $\psio=\sup C$.
For $0<\theta\notin C^+$, $g(\theta)=\infty$.
For $0<\theta\in C^+$,
%
%
\begin{equation}\label{form1}
\frac{g(\theta)}{\theta}
=\inf\biggl\{
\max\biggl\{\frac{f(\phi)}
{\phi},\frac{\kappa(\theta)}{\theta}\biggr\}
\dvtx
0<\phi\leq\theta, \phi\leq\psio\biggr\},
\end{equation}
where the condition $\phi\leq\psio$ can be omitted when (\ref
{eitheror})
holds and $\nat{f}$ is continuous from the right at $\psio$.
\end{lem}
\begin{pf} It is immediate from its definition that $g(\theta
)=\infty$
for $0<\theta\notin C^+$.
By definition
$\nat{f}(\theta)/\theta$ is decreasing as $\theta$ increases
for any convex $f$.
For $\theta\in C^+$,
%
%
\begin{eqnarray}\label{withC}
\frac{g(\theta)}{\theta}
&=&
\max\biggl\{\frac{\nat{(\nat{f}+\chi)}(\theta)}{\theta},\frac
{\kappa
(\theta)}{\theta}\biggr\}
\nonumber\\
&=&
\inf\biggl\{
\max\biggl\{\frac{\nat{(\nat{f}+\chi)}(\phi)}
{\phi},\frac{\kappa(\theta)}{\theta}\biggr\}
\dvtx0 < \phi\leq\theta\biggr\}
\nonumber\\
&=&
\inf\biggl\{
\max\biggl\{\frac{\nat{f}(\phi)}
{\phi},\frac{\kappa(\theta)}{\theta}\biggr\}
\dvtx0<\phi\leq\theta, \phi\in C\biggr\}.
\end{eqnarray}
Proposition~\ref{f-natural2} relates $\nat{f}$ and $f$:
$\nat{f}(\theta)/\theta$ and
$f(\theta)/\theta$ agree and are decreasing up to
$\varthetapp{f}$;
when $\varthetapp{f}<\infty$,
the former is constant and the latter is larger
for $\theta>\varthetapp{f}$, and
either the two agree at $\theta=\varthetapp{f}$
or the latter is larger. Hence,
\[
\frac{g(\theta)}{\theta}=
\inf\biggl\{
\max\biggl\{\frac{f(\varphi)}
{\varphi},\frac{\kappa(\theta)}{\theta}\biggr\}
\dvtx0<\phi\leq\theta, \varphi\leq\phi\in C\biggr\}.
\]
This is~(\ref{form1}) when $\psio\in C$. When it is not, the limit of
$ f(\varphi)/ \varphi$ as $\varphi\uparrow\psio$ is no greater than
$f(\psio)/ \psio$ and so replacing $\varphi\leq\phi\in C$
by $\varphi\leq\psio$ in the formula
will not change the output.

Lemma~\ref{fgagreelemma}\ref{fgagreelemmaii} shows that
if~(\ref{eitheror}) holds
and $\nat{f}$ is continuous from the right at~$\psio$,
then the restriction to $\phi\in C$ in~(\ref{withC})
can be replaced by $\phi\in C^+$. Then
$f$ can replace $\nat{f}$
if this restriction is dropped, too; that is, for $\theta\in C^+$,
\begin{eqnarray*}
\frac{g(\theta)}{\theta}
&=&
\inf\biggl\{
\max\biggl\{\frac{\nat{f}(\phi)}
{\phi},\frac{\kappa(\theta)}{\theta}\biggr\}
\dvtx0<\phi\leq\theta, \phi\in C^+ \biggr\}
\\
&=&
\inf\biggl\{
\max\biggl\{\frac{f(\phi)}
{\phi},\frac{\kappa(\theta)}{\theta}\biggr\}
\dvtx0<\phi\leq\theta\biggr\}.
\end{eqnarray*}
\upqed\end{pf}
\begin{pf*}{Proof of Theorem~\ref{gformullem1}}
The result is true for $K=1$ as is the additional condition that
$\Gamma(g^{\ast}_1)>-\infty$.
Assume it is true along with this additional condition for $K - 1$.
Let $\btheta=(\theta_1, \theta_2,
\ldots, \theta_{K - 1})$, $h(\btheta)=\max\{\kappa
_i(\theta_i)
/ \theta_i\dvtx\break{i\leq K - 1}\}$ and let $\Delta_\phi$ be
the set the infimum is taken over in~(\ref{gform1}) for ``$K - 1$''
so that the induction hypothesis is
\[
\frac
{g_{K - 1}(\phi)}
{\phi}=\inf\{h(\btheta)\dvtx
\btheta\in\Delta_\phi
\}.
\]
By the previous lemma, for $0<\theta\in\PhiD^+_{K - 1,K}$,
\[
\frac
{g_K(\theta)}
{\theta}=\inf\biggl\{
\max\biggl\{\frac{g_{K - 1}(\phi)}
{\phi},\frac{\kappa_K(\theta)}{\theta}\biggr\}
\dvtx
0<\phi\leq\theta,
\phi
\leq\psio_{K - 1} \biggr\}.
\]
Now
\[
\max\biggl\{\frac{g_{K - 1}(\phi)}
{\phi},\frac{\kappa_K(\theta)}{\theta}\biggr\}
=\max\biggl\{\inf\{h(\btheta)
\dvtx\btheta\in\Delta_\phi
\},\frac{\kappa_K(\theta)}{\theta}\biggr\}
\]
and reordering the
maximum and infimum on the right makes no difference.
This gives $g_K$ in the required form and
Lemma~\ref{fgagreelemma}\ref{fgagreelemmanewi} gives that
$\Gamma(g^{\ast}_K)>-\infty$, completing the induction.
Then the formula for $\Gamma(g_K)$
is, by Proposition~\ref{f-natural2}, obtained by
minimizing also over $\theta$. The result for $f_K$ is just a special case.
\end{pf*}
%
%
\begin{lem}\label{gformullem2}
Assume~(\ref{newphiiorder}) holds.
In~(\ref{gform1}) and~(\ref{speed-g}) the conditions
``$\theta_{i} \leq\psio_{i}$''
can be dropped if
(\ref{eitheror1}) holds for $
i=1,2,\ldots,K - 1$. The conditions
``$\theta_{i} \in\PhiD^+_{i-1,i}$'' can be\vspace*{1pt} dropped\vadjust{\goodbreak} in~(\ref{gform1}) if
$\varthetapp{\kappa_{i+1}} \geq\psiu_{ i}$
for $i=1,\ldots,K - 2$
and from~(\ref{speed-g}) if this holds
also for $i=K - 1$.
When both sets of conditions in~(\ref{speed-g})
can be dropped,
$\Gamma(g^{\ast}_K)
=\Gamma(f^{\ast}_K)$.
\end{lem}
\begin{pf} Lemma~\ref{niceg} gives that $\nat{g}_i$ is continuous
at $\psio_i$.
Then the proof that the conditions
$\theta_{i} \leq\psio_{i}$
can be dropped in
(\ref{gform1}) is by induction
on $i$ using the last
part of Lemma~\ref{gtot}.
When $\varthetapp{\kappa_{i+1}} \geq\psiu_{ i}$ for $i=1, \ldots,
K
- 2$
the extra possibilities included by discarding the conditions
$\theta_{i} \in\PhiD^+_{i-1,i}$
for $i=2, \ldots, K - 1$
in~(\ref{gform1}) are larger
than those included and so make no difference to
the infimum. (Here
$\theta_{K} \in\PhiD^+_{K - 1,K}$ cannot
be excluded, since the infimum is not over $\theta_K$.) The argument
simplifying~(\ref{speed-g}) is the same.
\end{pf}
\begin{pf*}{Proof of Theorem~\ref{bestspeed}}
This is contained in
Lemma~\ref{gformullem2}.
\end{pf*}

\section{Simplifying the formula for the speed}\label{simplespeed}

%
\begin{lem}\label{conccut}
Assume $f$ and $\kappa$ are $k$-convex, that $\cl{f}(0)>0$ and that
$g={\mathfrak M}[\nat{f},\kappa]$
is finite somewhere. Let $\varthetapps{g}=\varthetapp{g}$ [and, for
later, $\speedps=\Gamma(g^{\ast})$]. Then the following hold:
{\renewcommand\thelonglist{(\roman{longlist})}
\renewcommand\labellonglist{\thelonglist}
\begin{longlist}
\item\label{conccut-gnat}
$\nat{g} \geq{\mathfrak M}[\nat{f},\nat{\kappa}]$;
\item\label{conccut-kappa}
$\varthetapp{\kappa} \leq\varthetapps{g}$;
\item\label{conccut-g}
$
g(\theta)=\nat{g}(\theta)={\mathfrak M}[\nat{f},\nat{\kappa
}](\theta) $ for $\theta< \varthetapps{g}$.
\end{longlist}}
\end{lem}
\begin{pf}
Let $
\varphi= \inf\{\theta\dvtx\kappa(\theta)> {\mathfrak M}[\nat
{f},\nat{\kappa}](\theta) \}$. Observe that
\[
{\mathfrak M}[\nat{f},\kappa]=g \geq\nat{g}=\nat{ {\mathfrak
M}[\nat{f},\kappa]} \geq
\nat{ {\mathfrak M}[\nat{f},\nat{\kappa}]}=
{\mathfrak M}[\nat{f},\nat{\kappa}],
\]
where the final equality is from Lemma~\ref{fnatsimple}\ref
{sweeplem}, which gives~\ref{conccut-gnat}. There is equality throughout
when $\theta< \varthetapp{\kappa}$,
since then $\nat{\kappa}(\theta)=\kappa(\theta)$, and also when
$\theta< \varphi$.
This implies that $\varthetapp{\kappa}\leq\varthetapps{g}$,
proving~\ref{conccut-kappa}, and that $\varphi\leq\varthetapps{g}$.
Note too, for later in the proof,
that $\varthetapp{\kappa} \leq\varphi$, because $\nat{\kappa}$ and
$\kappa$ agree for $\theta< \varthetapp{\kappa}$.
It remains to show that $\varthetapps{g}\leq\varphi$.
It is certainly true that $\varthetapps{g}\leq\varphi$ when $\varphi
=\infty$.
Also if $\kappa(\theta)=\infty$ for all $\theta>\varphi$, then
$
g(\theta)=
\infty$ for $\theta> \varphi$,
but, by Proposition~\ref{f-natural2}, $\nat{g}$ is finite for
$\theta
>\varphi$
and so $\varthetapps{g}\leq\varphi$.
In the remaining case $\varphi<\infty$,
$\kappa$ is finite on
$(\varphi,\varphi+\varepsilon)$ for some $\varepsilon>0$, and there are
$\theta_i \downarrow\varphi$ taken from this interval
with
$
g(\theta_i)=\kappa(\theta_i)$.
By Lemma~\ref{subdiffeasy}\ref{subdiffeasyi} $\subdiff\kappa
(\theta_i)$ is nonempty. Hence, by
Lemma~\ref{rateagreenew}, $g^{\ast}(a)=\kappa^{\ast}(a)$ for
$a \in\subdiff\kappa(\theta_i)$. Since
$\varthetapp{\kappa}\leq\varphi$,
Lemma~\ref{f-fltiiii}\ref{f-fltiiiii} implies that
$\kappa^{\ast}(a)>0$. Hence $g^{\ast}(a)>0$
and Lemma~\ref{f-fltiiii}\ref{f-fltiiiiii} gives
$\varthetapps{g}\leq\varphi$.
\end{pf}
%
%
\begin{lem}\label{conccut2}
Use the setup of Lemma~\ref{conccut}.
{\renewcommand\thelonglist{(\roman{longlist})}
\renewcommand\labellonglist{\thelonglist}
\begin{longlist}
\item\label{conccut-i}
If
$
\speedps= \max\{\Gamma(f^{\ast}),\Gamma(\kappa^{\ast}) \},
$ then $\nat{g}(\theta)={\mathfrak M}[\nat{f},\nat{\kappa}](\theta
)$ except
possibly at $\theta=\varthetapps{g}$.
\item\label{conccut-ii} If
$
\speedps> \max\{\Gamma(f^{\ast}),\Gamma(\kappa^{\ast}) \}$, then
$\varthetapps{g}< \infty$, and
$({\mathfrak M}[\nat{f},\nat{\kappa}](\theta)-\theta\speedps)$
is strictly positive when $\theta< \varthetapps{g}$ and strictly negative
when $\theta> \varthetapps{g}$.\vadjust{\goodbreak}
\end{longlist}}
\end{lem}
\begin{pf}
Lemma~\ref{conccut2}\ref{conccut-g}
gives $\nat{g}(\theta)=g(\theta)={\mathfrak M}[\nat{f},\nat
{\kappa}](\theta)$
for $\theta< \varthetapps{g}$.
Assume that $
\speedps= \max\{\Gamma(f^{\ast}),\Gamma(\kappa^{\ast}) \}
$ and that $\varthetapps{g}<\infty$.
Then
Proposition~\ref{f-natural2} implies that
$
\nat{g}(\theta)= \theta\speedps$
for $\theta> \varthetapps{g}
$.
Similarly,
$
\nat{\kappa}(\theta)= \theta\Gamma(\kappa^{\ast})$
for $\theta> \max\{0,\varthetapp{\kappa}\}$.
If $\speedps=\Gamma(\kappa^{\ast})$,
$\nat{g}$ and $\nat{\kappa}$ agree for $\theta>\varthetapps{g}$.
If instead, $\speedps=\Gamma(f^{\ast})>\Gamma(\kappa^{\ast})$,
then, for $\theta> \varthetapps{g}$,
$
\nat{f}(\theta) \geq
\theta\Gamma(f^{\ast})
=\nat{g}(\theta)$.
Hence, in both cases, using also Lemma~\ref{conccut}\ref
{conccut-gnat}, $\nat{g}(\theta)={\mathfrak M}[\nat{f},\nat{\kappa
}](\theta)$
for $\theta> \varthetapps{g}$.

Assume now that $
\speedps> \max\{\Gamma(f^{\ast}),\Gamma(\kappa^{\ast}) \}
$.
Take $a$ such that
\[
\max\{\Gamma(f^{\ast}),\Gamma(\kappa^{\ast})\}< a<
\speedps.
\]
Using Lemma~\ref{kh}(ii) and the definition of $\Gamma(\cdot)$,
${\mathfrak M}[\nat{f},\nat{\kappa}]^{\ast}(a)=
{\mathfrak C}[f^{\acci},\kappa^{\acci}](a)=\infty$ and
$g^{\ast}(a)<0$.
Hence $g$ and ${\mathfrak M}[\nat{f},\nat{\kappa}]$
differ somewhere and so Lem\-ma~\ref{conccut}\ref{conccut-g}
implies that $\varthetapps{g}<\infty$.

Since $g(\theta)\geq\speedps\theta$
for all $\theta$, ${\mathfrak M}[\nat{f},\nat{\kappa}](\theta
)=g(\theta
)\geq
\speedps\theta$
for $\theta< \varthetapps{g}$ and
$\theta\speedps=\nat{g}(\theta)\geq{\mathfrak M}[\nat{f},\nat
{\kappa}](\theta
)$ for
$\theta>\varthetapps{g}$. It remains to show these inequalities
are strict.
Since ${\mathfrak M}[\nat{f},\nat{\kappa}](\theta)/\theta$ is decreasing
it can
only equal
$\speedps$ on an interval
that, if nonempty, includes $\varthetapps{g}$.
If the interval has a nonempty interior,
then, by convexity of ${\mathfrak M}[\nat{f},\nat{\kappa}]$,
${\mathfrak M}[\nat{f},\nat{\kappa}](\theta) \geq\speedps\theta$
for all
$\theta$,
contradicting that
${\mathfrak M}[\nat{f},\nat{\kappa}](\theta)/\theta\rightarrow
\max\{\Gamma(f^{\ast}),\Gamma(\kappa^{\ast})\}<\speedps$
as $\theta\rightarrow\infty$.
\end{pf}
%
%
\begin{lem}
\label{closed}
In the setup of Lemma~\ref{conccut} assume also that $\nat{f}$ and
$\kappa$ are closed.
{\renewcommand\thelonglist{(\roman{longlist})}
\renewcommand\labellonglist{\thelonglist}
\begin{longlist}
\item\label{closed-i}
If
$
\speedps= \max\{\Gamma(f^{\ast}),\Gamma(\kappa^{\ast}) \},
$ then $\nat{g}={\mathfrak M}[\nat{f},\nat{\kappa}]$.
\item\label{closed-ii} If
$
\speedps> \max\{\Gamma(f^{\ast}),\Gamma(\kappa^{\ast}) \}$, then
$\nat
{g}(\theta)=\theta\speedps$ when $\theta\geq\varthetapps{g}$ and
$\nat
{g}(\theta)={\mathfrak M}[\nat{f},\nat{\kappa}](\theta)$ when
$\theta<
\varthetapps{g}$.
\end{longlist}}
\end{lem}
\begin{pf} When $\nat{f}$ and $\kappa$ are closed so are $\nat
{\kappa
}$, $g$, $\nat{g}$ and ${\mathfrak M}[\nat{f},\nat{\kappa}]$.
Part~\ref{closed-i} now follows from Lemma~\ref{conccut2}\ref{conccut-i} and
part~\ref{closed-ii} from Proposition~\ref{f-natural2} and Lem\-ma~\ref
{conccut}\ref{conccut-g}.
\end{pf}
%
%
\begin{lem}\label{conccut21}
In the setup of Lemma~\ref{conccut},
assume
$
\speedps> \max\{\Gamma(f^{\ast}), \Gamma(\kappa^{\ast}) \}
$. Then $g(\theta)=\kappa(\theta)>\nat{f}(\theta)$ on
$(\varthetapps
{g}, \infty)$.
{\renewcommand\thelonglist{(\roman{longlist})}
\renewcommand\labellonglist{\thelonglist}
\begin{longlist}
\item\label{conccut21i}If $\PhiD( \kappa)=\{\phi\}$, then
$\varthetapps{g}=\phi$,
$\kappa(\varthetapps{g})<\nat{f}(\varthetapps{g})
=g(\varthetapps{g})<\infty$ and $g$ is infinite elsewhere.
\item\label{conccut21ii}If $\PhiD( \kappa)$ is not a
single point, then,
for some $\varepsilon>0$,
$g(\theta)=\nat{f}(\theta)>\kappa(\theta)$ on
$(\varthetapps{g}-\varepsilon,\varthetapps{g})$.
\end{longlist}}
\end{lem}
\begin{pf} Using the definition of $g$ and Lemma~\ref{conccut2}\ref
{conccut-ii},
\[
{\mathfrak M}[\nat{f},\kappa]=g(\theta)>\nat{g}(\theta)=\Gamma
\theta>
{\mathfrak M}[\nat{f},\nat{\kappa}] \qquad\mbox{for }\theta\in
(\varthetapps{g},\infty).
\]
Thus $g$ agrees with $\kappa$ and strictly exceeds $\nat{f}$
on $(\varthetapps{g},\infty)$.

If $\varthetapps{g} =\inf\PhiD( \kappa) < \sup\PhiD
( \kappa)$,
then the closures of $g$ and $\kappa$
agree everywhere, giving $\speedps=\Gamma(\kappa^{\ast})$, which has
been ruled out.
Hence either $\PhiD( \kappa)=\{\varthetapps{g}\}$ and
$\kappa(\varthetapps{g})<\nat{f}(\varthetapps{g})$, giving \ref
{conccut21i},
or $\inf\PhiD( \kappa)<\varthetapps{g}\leq\sup\PhiD
( \kappa)$.
Assume the latter, so that there is an $\varepsilon>0$ such that
$\kappa$ is finite, and continuous,
on $(\varthetapps{g}-\varepsilon,\varthetapps{g})$ and so
$\nat{\kappa}$ is finite and continuous on $(\varthetapps
{g}-\varepsilon
,\infty)$. When
$\nat{f}$
is infinite on $(-\infty,\varthetapps{g})$ the result holds.
Hence by adjusting $\varepsilon$, we can now assume $\nat{f}$
is also finite on $(\varthetapps{g}-\varepsilon,\infty)$.
Say $\varthetapp{\kappa}=\varthetapps{g}$. Using continuity on
$(\varthetapps{g}-\varepsilon,\infty)$, Proposition~\ref{f-natural2} and
Lemma~\ref{conccut}\ref{conccut-g},
\[
\speedps\varthetapps{g} = \nat{g}(\varthetapps{g})= \max\{\nat
{f}(\varthetapps{g}), \nat{\kappa}(\varthetapps{g})\}
>\Gamma(\kappa^{\ast})\varthetapps{g}=\nat{\kappa}(\varthetapps{g}).
\]
A further use of continuity now gives $\nat{f}(\theta)>\nat{\kappa
}(\theta)=\kappa(\theta)$ on $(\varthetapps{g}-\varepsilon
,\varthetapps{g})$
after, if necessary,
taking $\varepsilon$ smaller. This proves~\ref{conccut21ii} in this case.

Say now that $\varthetapp{\kappa}<\varthetapps{g}$, which by Lemma
\ref{conccut}\ref{conccut-kappa}
is the only other possibility, and adjust $\varepsilon$ so that
$\varthetapp{\kappa}\leq\varthetapps{g}-\varepsilon$.
Suppose, for a contradiction, that there is a
$\psi\in(\varthetapps{g}-\varepsilon,\varthetapps{g})$
with $\kappa(\psi)=g(\psi)$.
Take $a \in\subdiff\kappa(\psi)$, which is nonempty.
By Lem\-ma~\ref{f-fltiiii}\ref{f-fltiiiii}, $\kappa^{\ast}(a)>0$
because $\psi>\varthetapp{\kappa}$, but $g \geq\kappa$ and so Lemma
\ref{rateagreenew}
gives $\kappa^{\ast}(a)=g^{\ast}(a)$. However, by Lemma \ref
{f-fltiiii}\ref{f-fltiiiiii},
$\psi<\varthetapps{g}$ implies $g^{\ast}(a)\leq0$.
Hence there is no such $\psi$ and so $g=\nat{f}>\kappa$
on $(\varthetapps{g}-\varepsilon,
\varthetapps{g})$.
\end{pf}
%
%
\begin{lem}\label{indulem} In the setup and conditions of Proposition
\ref{alternativerec}, suppose that $\kappa_1(0)>0$ and that
$\Gamma(f^{\ast}_K)>\max\{\Gamma(f^{\ast}_{K - 1}),
\Gamma(\kappa^{\ast}_K)\}$.
Then
\[
f_K=
{\mathfrak M}\Bigl[\max_{j} \nat{\kappa}_j,\kappa_K \Bigr].
\]
\end{lem}
\begin{pf}
For $i=1,2,\ldots, K$, let
\[
h_{i}=
{\mathfrak M}\Bigl[\max_{j \geq i} \nat{\kappa}_{j},\kappa_K \Bigr]
\]
so that $h_{K}=\kappa_{K}$.
Now suppose that
%
%
\begin{equation}\label{industart}
f_K={\mathfrak M}[\nat{f}_{i},h_{i+1}],
\end{equation}
which is true, by definition, for $i=K - 1$.
Induction will be used to show that this holds also for $i=1$, which is
the required result because $\nat{f}_{1}=\nat{\kappa}_{1}$.

Assume~(\ref{industart}) holds for $i$ and
consider $\nat{f}_{i}=\nat{{\mathfrak M}[\nat{f}_{i-1},\kappa
_{i}]}$. Using
Lemmas~\ref{analytic},~\ref{alternativerectwo} and~\ref{closed},
there are two possibilities.
One is that $\nat{f}_{i}={\mathfrak M}[\nat{f}_{i-1},\nat{\kappa}_{i}]$
everywhere, in which case,
%
%
\begin{equation}\label{indu}
f_K=\max
\{
\nat{f}_{i-1},\nat{\kappa}_{i}, h_{i+1}
\}
={\mathfrak M}[\nat{f}_{i-1},h_{i}],
\end{equation}
giving~(\ref{industart}) for $i-1$.
Otherwise, $\varthetapp{f_i}<\infty$ and
\[
\nat{f}_{i}(\theta)=\cases{
\theta\Gamma(f^{\ast}_i), &\quad for $\theta\geq\varthetapp
{f_i}$,\cr
{\mathfrak M}[\nat{f}_{i-1},\nat{\kappa}_{i}](\theta),
&\quad for $\theta< \varthetapp{f_i}$.}
\]
Thus~(\ref{indu}) holds for $\theta< \varthetapp{f_i}$.
Also, $\Gamma(f^{\ast}_i)\leq\Gamma(f^{\ast}_{K - 1})<\Gamma
(f^{\ast}_K)$, which implies that $f^{\ast}_{K}(\Gamma(f^{\ast
}_{i}))<0$. Hence,
for all $\theta$,
$
\theta\Gamma(f^{\ast}_i)<f_K(\theta)
$
and so, in particular, when $\theta\geq\varthetapp{f_i}$
\[
f_K(\theta)={\mathfrak M}[\nat{f}_{i},h_{i+1}](\theta)
=
\max\{\theta\Gamma(f^{\ast}_i),h_{i+1}(\theta) \}
=
h_{i+1}(\theta).\vadjust{\goodbreak}
\]
Thus, using this and Lemma~\ref{fnatsimple}\ref{fnatsimpleiii},
\[
h_{i+1}(\theta)> \theta\Gamma(f^{\ast}_i)=\nat{f}_{i}(\theta)
\geq{\mathfrak M}[\nat{f}_{i-1},\nat{\kappa}_{i}](\theta).
\]
Hence,~(\ref{indu}) also holds when $\theta\geq\varthetapp{f_i}$.
This shows that~(\ref{indu})
always holds when~(\ref{industart}) holds, which completes the
inductive step.
\end{pf}
%
%
\begin{lem}\label{pairwiselem} In a sequential process satisfying
$\kappa_1(0)>0$ and
(\ref{phiiorder}), let $r_K$ be
given by the recursion~(\ref{firstrecur}) described in Theorem \ref
{prelimmaintheorem}. Then
\[
\Gamma(r_K)= \max_{i \bef j} \{ \Gamma({\mathfrak
C}[\kappa^{\acci}_i,\kappa^{\ast}_j])\} = \max_{i \bef j} \{
\Gamma({\mathfrak M}[\nat{\kappa}_i,\kappa_j]^{\ast}) \}.
\]
\end{lem}
\begin{pf}
Note first that for a sequential process $i
\bef j$ is the same as $i < j$.
Take $f_i$ as in Proposition~\ref{alternativerec}, so that
$r_i=f^{\acci}_i=(\nat{f}_i)^{\ast}$.
Let $\speedps=\Gamma(r_K)$ $(\mbox{$=$}\Gamma(f^{\ast}_K))$ and
$\varthetapps
{f_K}=\varthetapp{f_K}$. Since $\Gamma(\kappa^{\ast}_K) \leq
\Gamma({\mathfrak C}[\kappa^{\acci}_1,\kappa^{\ast}_K])$, it would be
enough to establish
the result for $\Gamma(r_{K - 1})$ in the case where $\speedps=
\max\{\Gamma(r_{K - 1}),\Gamma(\kappa^{\ast}_K)\} $.
Consequently, we can assume that
$\speedps>
\max\{\Gamma(r_{K - 1}),\Gamma(\kappa^{\ast}_K)\} $.
Now, Lem\-ma~\ref{conccut2} gives
$\varthetapps{f_K}<\infty$, and $f^{\ast}_K(\speedps)\leq0$
implies that
$\speedps\theta\leq f_K(\theta)$
everywhere.

Let
\[
h=\max\{
\nat{\kappa}_j\dvtx j \leq K-1\}.
\]
If $h$ is infinite on $(-\infty,\varthetapps{f_K})$, then there is a
$J<K$ with $\nat{\kappa}_J$ infinite on $(-\infty,\varthetapps{f_K})$.
If $\PhiD( \kappa_K)=\{\varthetapps{f_K}\}$,
then, by Lemma~\ref{conccut21}\ref{conccut21i}, there is a $J<K$
with $\nat{\kappa}_J(\varthetapps{f_K})>\kappa_K(\varthetapps{f_K})$.
In both these cases Lemma~\ref{conccut21} implies that $
f_K={\mathfrak M}[\nat{\kappa}_J,\kappa_K]$ and so $\speedps=\Gamma
({\mathfrak M}[\nat{\kappa}_J,\kappa_K]^{\ast})$.
Otherwise, using Lemma~\ref{conccut21}\ref{conccut21ii}, there is
an $\varepsilon>0$ such that $h$ and $\kappa_K$ are finite and continuous
on $(\varthetapps{f_K}-\varepsilon,\varthetapps{f_K})$.
Now, suppose that
$h(\varthetapps{f_K})>\kappa_K(\varthetapps{f_K})$, and take $J<K$ with
$\kappa_J(\varthetapps{f_K})=h(\varthetapps{f_K})$.
Using the continuity of $\kappa_K$ when finite,
there is an $\varepsilon>0$ such that $\kappa_K(\theta)<\kappa
_J(\theta)$
on $(\varthetapps{f_K}-\varepsilon,\varthetapps{f_K})$.
Also, Lemma~\ref{conccut21} implies that $\kappa_K$ is infinite on
$(\varthetapps{f_K},\infty)$.
Therefore, since
$\nat{\kappa}_J(\theta)/\theta$ is decreasing in $\theta$,
%
%
\begin{equation}
\label{justapair}
\speedps= \inf_{\theta>0} \frac{f_K(\theta)}{\theta} \geq
\inf_{\theta>0}\frac{{\mathfrak M}[\nat{\kappa}_J,\kappa_K](\theta
)}{\theta}
=\frac{\nat{\kappa}_J(\varthetapps{f_K})}{\varthetapps{f_K}}=
\frac{f_K(\varthetapps{f_K})}{\varthetapps{f_K}}=\speedps
\end{equation}
and so again $\speedps=\Gamma({\mathfrak M}[\nat{\kappa}_J,\kappa
_K]^{\ast})$.

This leaves the case where, for some $\varepsilon>0$, $f_K$
is finite on $(\varthetapps{f_K}-\varepsilon,\varthetapps{f_K}]$ and
$h(\varthetapps{f_K}) \leq\kappa_K(\varthetapps{f_K})$. Then\vspace*{1pt}
$\nat{\kappa}_j$ is continuous on
$(\varthetapps{f_K}-\varepsilon,\infty)$ for every $j$
and thus by Lemma~\ref{conccut21},
$f_K(\theta)=h(\theta)> \kappa_K(\theta)$ on $(\varthetapps
{f_K}-\varepsilon,\varthetapps{f_K})$ and $f_K(\theta)=\kappa_K(\theta
)>h(\theta)$ on $(\varthetapps{f_K},\infty)$.
By continuity and Lemma~\ref{closed}\ref{closed-ii}, $h(\varthetapps
{f_K})=\kappa_K(\varthetapps{f_K})=\speedps\varthetapps{f_K}$.
Let $\cal I$ be those $j < K$ with
$\nat{\kappa}_j(\varthetapps{f_K})=\speedps\varthetapps{f_K}$
and let $\tilde{h}=\max\{\nat{\kappa}_j\dvtx j \in{\cal I}\}$.
By reducing $\varepsilon$ if necessary, $f_K=\tilde{h}>\kappa_K$ on
$(\varthetapps{f_K}-\varepsilon,\varthetapps{f_K})$.
Let
$\gamma_j= \inf\subdiff\nat{\kappa}_j(\varthetapps{f_K})$
and
take $J$ to be an index
giving $\min\{\gamma_j\dvtx j \in{\cal I}\}$.
Take $\varepsilon'>0$. Then,
for some $\delta>0$, for $\theta\in(\varthetapps{f_K}-\delta
,\varthetapps{f_K})$ and
$j \in\cal I$,
\[
\nat{\kappa}_j(\theta) \leq\nat{\kappa}_j(\varthetapps{f_K})+(\gamma
_j-\varepsilon') (\theta-\varthetapps{f_K})\qquad
\bigl(\mbox{$=$}\speedps\varthetapps{f_K}+(\gamma_j-\varepsilon')
(\theta-\varthetapps{f_K})\bigr)\vadjust{\goodbreak}
\]
for otherwise, by convexity,
$(\gamma_j-\varepsilon') \in\subdiff\nat{\kappa}_j(\varthetapps{f_K})$.
Then, taking the max of these over
$j \in{\cal I}$ with $\delta$ as the minimum of those needed
gives
\[
f_K(\theta)=\tilde{h}(\theta)
\leq\speedps\varthetapps{f_K}+ (\gamma_J-\varepsilon')
(\theta-\varthetapps{f_K})
\]
for
$\theta\in(\varthetapps{f_K}-\delta,\varthetapps{f_K})$. But
$\speedps\theta\leq f_K(\theta)$
everywhere. Hence
\[
(\gamma_J-\varepsilon')
(\varthetapps{f_K}-\theta) \leq\speedps(\varthetapps{f_K}-\theta)
\theta\in(\varthetapps{f_K}-\delta,\varthetapps{f_K})
\]
and so $\gamma_J \leq\speedps$.
Therefore, for $\theta\leq\varthetapps{f_K}$,
\[
f_K(\theta) \geq\nat{\kappa}_J(\theta) \geq
\speedps\varthetapps{f_K}+ \gamma_J
(\theta-\varthetapps{f_K})\geq
\speedps\theta
\]
and for $\theta> \varthetapps{f_K}$, $f_K(\theta)=\kappa_K(\theta)$
and is strictly greater than both $\kappa_J(\theta)$ and~$\speedps
\theta$.
Thus~(\ref{justapair}) holds in this case, too, giving
$\speedps=\Gamma({\mathfrak M}[\nat{\kappa}_J,\kappa_K]^{\ast})$.
\end{pf}
\begin{pf*}{Proof of Theorem~\ref{pairwiseformula}}
Applying Lemma~\ref{pairwiselem} to every sequential process
gives the first formula for $\speedps$. Fix $i \bef j$. Let $f=\kappa
_i$, $\kappa=\kappa_j$ and $g={\mathfrak M}[\nat{f},\kappa]$ so that
$\Gamma({\mathfrak C}[\kappa_i^{\acci},\kappa_j^{\ast}])
=
\Gamma(g^{\ast})$. Now, an application of Lemma~\ref{gtot} (with
$C=[0,\infty)$)
and then of~(\ref{w-speed}) in Proposition~\ref{f-natural2}
gives the second formula.
\end{pf*}

\section{Expected numbers}\label{expnumbers}
%
%
\begin{theorem}
Consider a sequential process
with $K$ classes, $\conC{1}, \ldots, \conC{K}$, with
corresponding \PF eigenvalues $\kappa_1,\ldots,\kappa_K$ and in which
$\conC{1}$ is primitive. Suppose that
%
%
\begin{equation}\label{moments}
\bigcap_{j \leq K} \PhiD( \kappa_{j})
\neq\varnothing\quad\mbox{and}\quad
\bigcap_{j \leq i+1} \PhiD( \kappa_{j}) \subset\PhiD_{i, i+1}
\qquad\mbox{for } i=1,\ldots, K - 1.\hspace*{-35pt}
\end{equation}
Define $R_i$ recursively by $R_1= \kappa^{\ast}_1$ and
$R_{i}={\mathfrak C}[R_{i-1},\kappa^{\ast}_i]$ for $i=2,\ldots, K$.
Then
%
%
\begin{equation}\label{expectkeyresult}
\frac{1}{n} \log\bigl(\tE_\nu Z^{(n)}_{\sigma}[na,\infty) \bigr)
\rightarrow-R_K(a)
\end{equation}
except possibly at the upper endpoint of the interval on which $R_K$ is finite.
\end{theorem}
\begin{pf}Suppose that $m_{\upsilon\tau}> 0$ for $\upsilon\in\pen
$ and $\tau\in\last$.
Then
\[
\int e^{\theta z} \tE_\nu Z^{(n)}_\sigma(dz)= \sum_{r=0}^{n-1}
(m(\theta)^r)_{\nu\upsilon}
m(\theta)_{\upsilon\tau}
(m(\theta)^{n-r-1})_{ \tau\sigma}
\]
and so, by induction on the number of classes,
\[
\frac{1}{n} \log\int e^{\theta z} \tE_\nu Z^{(n)}_\sigma(dz)
\rightarrow\max_i \{ \kappa_i (\theta)\} \qquad\mbox{for } \theta>0.
\]
The second part of~(\ref{moments}) ensures the off-diagonal terms have
no effect; the first part ensures that the limit here is finite for
some $\theta>0$.
Induction on the number of classes shows that $R_K$ is the F-dual of
$\max_i \{ \kappa_i (\theta)\}$.
Now, as in Proposition~\ref{behaviour1}, large deviation theory gives
(\ref{expectkeyresult}).\vadjust{\goodbreak}
\end{pf}

Although $R_K$ is defined recursively it can be defined directly as the
convex minorant of $\kappa^{\ast}_1, \ldots, \kappa^{\ast}_K$.
It is easy to see, by induction, that $r_i \geq R_i$, so that $\Gamma(r_K)
\leq\Gamma(R_K)$.
To see that $R_i$ and $r_i$
really can be different, notice that the order of the classes matters
in $r_i$ but does not in $R_i$. It is easy to give a two-type reducible
example where
$\Gamma(r_K)<\Gamma(R_K)$. More specifically, arrange $\kappa
^{\ast}_1$
and $\kappa^{\ast}_2$ so that:
{\renewcommand\thelonglist{(\roman{longlist})}
\renewcommand\labellonglist{\thelonglist}
\begin{longlist}
\item$\kappa^{\ast}_1(\Gamma)=\kappa^{\ast}_2(\Gamma)=0$,
\item\label{last-iii} $\kappa^{\ast}_1(x)<\kappa^{\ast}_2(x)$ for
$x>\Gamma$,
\item their convex minorant is less than zero at $\Gamma$.

Then in computing
$\Gamma(r_K)$, these last two conditions do not matter, and \mbox{$\Gamma
(r_K)=\Gamma$}. However, they do matter in computing $\Gamma(R_K)$
which will be bigger than $\Gamma$. Note too that, if instead of type 1
preceding type 2 here, type~2 preceded type 1, then $\Gamma
(r_K)=\Gamma(R_K)$ and this would be an example of super-speed, as
described toward the end of the \hyperref[intro]{Introduction}
and in \citet{MR2744237}.\looseness=-1
\end{longlist}}
%
\section{Further lower bounds}\label{flb}
Consider a sequential process with $m_{\upsilon\tau}> 0$ for
$\upsilon\in\pen$ and $\tau\in\last$.
Once either~(\ref{eitheror1}) or
(\ref{ii}) fails for $i=K - 1$,
the behavior of
$\tE_{\upsilon} Z_\tau[x,\infty)$ starts to exert an influence:
the spatial spread of the children in the final class (of type $\tau$)
born to a parent in the penultimate class (of type $\upsilon$) matters.
It seems that
some regularity is needed beyond knowledge of the interval of
convergence of
$m_{\upsilon\tau}$ to derive a result similar to
Theorem~\ref{maintheorem} in this case.
The conditions~(\ref{nicenearpsio}) and~(\ref{nicenearpsiu})
in the next result
are on the tails
of the distribution of average numbers of type $\tau$ born to a
type $\upsilon$.\looseness=-1
%
%
\begin{theorem}\label{refinedmaintheorem}
Make the same assumptions as in Theorem~\ref{prelimmaintheorem};
define $g_i$ by the recursion~(\ref{thirdrecur}) in Theorem \ref
{newub} and assume~(\ref{newphiiorder}) holds.
Let $\upsilon\in\pen$ and $\tau\in\last$ be the types for which
$m_{\upsilon\tau}\neq0$ and
let
\begin{eqnarray*}
\psio&=&\sup\{\psi\dvtx m_{\upsilon\tau}(\psi)<\infty\}=\sup\PhiD
_{K -
1,K},\\
\psiu&=&\inf\{\psi\dvtx m_{\upsilon\tau}(\psi)<\infty\}=\inf\PhiD
_{K - 1,K}.
\end{eqnarray*}
Assume also that
\[
\lim\frac{1}{n} \log
\bigl(Z^{(n)}_{\upsilon}[na,\infty)
\bigr) = -g^{\acci}_{K - 1}(a) \qquad\mbox{a.s.-}\tP_{ \nu}.
\]
Finally, assume both of the following: if~(\ref{eitheror1})
fails for $i=K - 1$, then
%
%
\begin{equation}\label{nicenearpsio}
\lim_{x\rightarrow\infty}
\frac{\log\tE_{\upsilon} Z_\tau[x,\infty)}{x} = - \psio;
\end{equation}
if~(\ref{ii}) fails for $i=K - 1$, then
%
%
\begin{equation}\label{nicenearpsiu}
\lim_{x\rightarrow-\infty}
\frac{\log\tE_{\upsilon} Z_\tau[x,\infty)}{x} = - \psiu.
\end{equation}
Then
\[
\lim\frac{1}{n} \log
\bigl(Z^{(n)}_{\sigma}[na,\infty)
\bigr) = -g^{\acci}_{K}(a) \qquad\mbox{a.s.-}\tP_{ \nu}.\vadjust{\goodbreak}
\]
\end{theorem}

Note that Kawata [(\citeyear{MR0464353}), Theorem 7.7.4] shows that the
$\limsup$ of the sequences in~(\ref{nicenearpsio}) and
(\ref{nicenearpsiu}) must be $-\psio$ and $-\psiu$,
respectively. Thus, the substance of each condition is that the lim inf
equals the corresponding lim sup.
This theorem
improves on the lower bound in Theorem~\ref{prelimmaintheorem} in some
cases, and matches the upper bound already obtained. It is not too hard
to obtain with the machinery already established.
%
%
\begin{lem}\label{refinedmainlemma} In a sequential process,
let $\upsilon\in\pen$, $\tau\in\last$, $\psio$ and
$\psiu$ as in Theorem~\ref{refinedmaintheorem}
and suppose that for $\nu\in\first$, and $k$-convex $f$ with
$\Gamma(f^{\ast})>-\infty$,
\[
\lim\frac{1}{n} \log
\bigl(Z^{(n)}_{\upsilon}[na,\infty)
\bigr) = -f^{\acci}(a) \qquad\mbox{a.s.-}\tP_{ \nu}
\]
for $a \neq\Gamma(f^{\ast})$.
Let $\chi_1(\theta)=-\log I(\theta\in[\psiu, \infty))$
and $\chi_2(\theta)=-\log I(\theta\in(-\infty,\break \psio])$.
Then
\[
\liminf\frac{1}{n} \log
\bigl(F^{(n)}_{\tau}[na,\infty)
\bigr) \geq-g^{\acci}(a) \qquad\mbox{a.s.-}\tP_{ \nu}
\]
for all
$a < \Gamma(g^{\ast})$, where:
\begin{longlist}
\item
$g=\nat{f}$ or
\item$g=\nat{f}+\chi_2$ when~(\ref{nicenearpsio}) holds, or
\item$g=\nat{f}+\chi_1$ when~(\ref{nicenearpsiu}) holds, or
\item
$g=\nat{f}+\chi_1+\chi_2$ when both~(\ref{nicenearpsio}) and
(\ref
{nicenearpsiu}) hold.
\end{longlist}
\end{lem}
\begin{pf}
Case (i) is given by Proposition~\ref{inductthm-a1}.
Let $C= \PhiD_{K - 1,K}$. Case (iv) is considered;
the other two are similar.
Assume $\nat{f}(\theta)<\infty$ for
some $\theta< \psiu$ and that $\psio<\infty$;
otherwise this is equivalent to cases (ii)
or (iii).
Then
\[
g^{\ast}(a)=\sup_{\theta\in C}\{\theta a-\nat{f}(a)\}=\sup_{\theta
\in
C}\{\theta a-f^{\flat}(a)\}.
\]
Let
\[
\underline{\gamma} = \inf\{\gamma'\dvtx\gamma' \in\subdiff f^{\flat
}(\theta
), \theta\in C\}
\]
and let $\overline{\gamma}$ be the supremum over the same set: both
are finite.
Calculations like those in Lemma~\ref{rateagreenew} show that
\[
g^{\ast}(a)=
\cases{
\psiu a -f^{\flat}(\psiu),&\quad$a\in(-\infty,\underline{\gamma}]$,\vspace*{2pt}\cr
f^{\acci}(a),&\quad$a \in(\underline{\gamma} ,\overline{\gamma})$, \vspace*{1pt}\cr
\psio a -f^{\flat}(\psio), &\quad$a\in[\overline{\gamma} ,\infty)$.}
\]

The number to the right of $nc$ in generation $n$ exceeds
$N_n=Z^{(n-1)}_{\upsilon}[n a,\infty)$
independent copies of $Z_{\tau}[n(c-a),\infty)$
under $\tP_{ \upsilon}$.
Let the expectation of the latter be $\tilde{e}_n$.\vadjust{\goodbreak}
Here $a<c$, since $n(c-a)$ must go to infinity, but otherwise $a$ may
be chosen freely.
When $f^{\acci}(a)<0$, Lemma~\ref{bin-ld2} and~(\ref{nicenearpsio}) give
\begin{eqnarray*}
\liminf_n \frac{1}{n}\log\tE\bigl[Z^{(n)}_{\tau}[n c,\infty)| \cF{n-1}\bigr]
&\geq& \liminf_n \frac{1}{n}(\log N_n + \log\tilde{e}_n)
\\
&\geq& - \bigl(f^{\acci}(a) +\psio(c-a)\bigr)
\end{eqnarray*}
and so, maximizing over the available $a$,
\[
\liminf_n \frac{1}{n}\log\tE\bigl[Z^{(n)}_{\tau}[n c,\infty)| \cF{n-1}\bigr] \geq
\sup_{f^{\acci}(a)<0, a<c}\{\psio a -f^{\acci}(a)\} -\psio c.
\]
Since $f^{\acci}$ is closed, increasing and infinite when positive, $\{
f^{\acci}(a)<0,a<c\}$
may be replaced by $\{a\leq c\}$. Then
using Lemmas~\ref{subdiffeasy} and~\ref{rateagreenew}
\[
\liminf_n
\frac{1}{n}\log\tE\bigl[Z^{(n)}_{\tau}[n c,\infty)| \cF{n-1}\bigr] \geq
\cases{ f^{\flat}(\psio)-\psio c,&\quad for $c
\geq\overline{\gamma}$,\cr-f^{\acci}(c), &\quad for $c <
\overline{\gamma}$,}
\]
when this is strictly positive.
Similarly, but with
$a>c$, so that $n(c-a)$ goes to minus infinity,
\begin{eqnarray*}
\liminf_n
\frac{1}{n}\log\tE\bigl[Z^{(n)}_{\tau}[n c,\infty)| \cF{n-1}\bigr]
&\geq&
\liminf_n \frac{1}{n}(\log N_n + \log\tilde{e}_n)
\\
&\geq&
- \bigl(f^{\acci}(a) +\psiu(c-a)\bigr)
\end{eqnarray*}
provided the latter is strictly positive. Then,
maximizing over $a>c$,
\[
\liminf_n
\frac{1}{n}\log\tE\bigl[Z^{(n)}_{\sigma}[n c,\infty)| \cF{n-1}\bigr]
\geq
\cases{f^{\flat}(\psiu) - \psiu c,&\quad for
$c\leq\underline{\gamma}$,\vspace*{2pt}\cr
-f^{\acci}(c),&\quad for $c >\underline{\gamma}$,}
\]
again, provided the latter is strictly positive.

Combining these,
\[
\liminf_n
\frac{1}{n}\log\tE\bigl[Z^{(n)}_{\sigma}[n c,\infty)| \cF{n-1}\bigr]
\geq-
g^{\ast}(c),
\]
when this is strictly positive.
Then conditional Borel--Cantelli and continuity of $g^{\acci}$ complete
the proof.
\end{pf}
\begin{pf*}{Proof of Theorem~\ref{refinedmaintheorem}}
First apply Lemma~\ref{fgagreelemma} to
determine which of the four possibilities
in Lemma~\ref{refinedmainlemma} is relevant.
Now use Lemma~\ref{refinedmainlemma} to show
\[
\liminf\frac{1}{n} \log
\bigl(F^{(n)}_{\tau}[na,\infty)
\bigr) \geq-g^{\acci}_{K - 1}(a) \qquad\mbox{a.s.-}\tP_{ \nu},
\]
and then use Theorem~\ref{inductthm} to complete the proof.
\end{pf*}


%

\printaddresses

\end{document}